\documentclass[11pt]{amsart}
\usepackage{amssymb,mathrsfs,graphicx,enumerate,mathabx}
\usepackage{amsmath,amsfonts,amssymb,amscd,amsthm,bbm}
\usepackage{mathtools}
\usepackage[retainorgcmds]{IEEEtrantools}
\usepackage{colortbl}
\usepackage{lipsum}
\usepackage{graphicx, subfigure}
\usepackage[caption=false]{subfig}
\usepackage{graphicx}
\title[]{Aggregation-diffusion energies on Cartan-Hadamard manifolds of unbounded curvature}

\author[Fetecau]{Razvan C. Fetecau}
\address[Razvan C. Fetecau]{\newline Department of Mathematics, Simon Fraser University, 8888 University Dr., Burnaby, BC V5A 1S6, Canada}
\email{van@math.sfu.ca}

\author[Park]{Hansol Park}
\address[Hansol Park]{\newline Department of Mathematics, Simon Fraser University, 8888 University Dr., Burnaby, BC V5A 1S6, Canada}
\email{hansol\_park@sfu.ca}

\newtheorem{theorem}{Theorem}[section]
\newtheorem{lemma}{Lemma}[section]
\newtheorem{corollary}{Corollary}[section]
\newtheorem{proposition}{Proposition}[section]
\newtheorem{remark}{Remark}[section]
\newtheorem{example}{Example}[section]

\newcommand{\bbr}{\mathbb R}

\newcommand{\bbs}{\mathbb S}

\newcommand{\bbh}{\mathbb H}

\newcommand{\calA}{\mathcal{A}}
\newcommand{\calC}{\mathcal{C}}

\newcommand{\calP}{\mathcal{P}}
\newcommand{\calR}{\mathcal{R}}

\newcommand{\calK}{\mathcal{K}}
\newcommand{\calW}{\mathcal{W}}

\newcommand{\calJ}{\mathcal{J}}
\def\d{\mathrm{d}}

\newcommand{\dm}{n} 

\newcommand{\p}{o}
\newcommand{\h}{h}
\newcommand{\dist}{d}
\newcommand{\secc}{\mathcal{K}}

\makeatletter
\@namedef{subjclassname@2020}{\textup{2020} Mathematics Subject Classification}
\makeatother
\begin{document}

\date{\today}



\begin{abstract} 
We consider an aggregation-diffusion energy on Cartan-Hadamard manifolds with sectional curvatures that can grow unbounded at infinity. The energy corresponds to a macroscopic aggregation model that involves nonlocal interactions and linear diffusion. We establish necessary and sufficient conditions on the growth at infinity of the attractive interaction potential for ground states to exist. Specifically, we derive explicit conditions on the attractive potential in terms of the bounds on the sectional curvatures at infinity.  To prove our results we establish a new logarithmic Hardy-Littlewood inequality for Cartan-Hadamard manifolds of unbounded curvature.
\end{abstract}

\maketitle \centerline{\date}

\section{Introduction}
\label{sect:intro}

In this paper we investigate the existence of ground states of the {\em free energy}
\begin{equation}
\label{eqn:energy-s}
E[\rho]=\int_M \rho(x)\log\rho(x)\d x+\frac{1}{2}\iint_{M \times M} W(x, y)\rho(x)\rho(y)\d x \d y,
\end{equation}
defined for probability measures on $M$, where $M$ is a general Cartan-Hadamard manifold of finite dimension, and $W: M \times M \to \bbr$ is an interaction potential -- see \eqref{eqn:energy} for the precise definition. Here, the integration is with respect to the Riemannian volume measure $\d x$.

In terms of dynamical evolution, the gradient flow of the energy $E$ on a suitable Wasserstein space results in the following nonlinear nonlocal evolution equation \cite{AGS2005}:
\begin{equation}
\label{eqn:model}
\partial_t\rho(x)- \nabla_M \cdot(\rho(x)\nabla_M W*\rho(x))=\Delta \rho(x),
\end{equation}
where
\[
W*\rho(x)=\int_M W(x, y)\rho(y) \d y,
\]
and $\nabla_M \cdot$ and $\nabla_M $ represent the manifold divergence and gradient, respectively. Consequently, critical points of the energy $E$ correspond to steady states of \eqref{eqn:model}.

In applications, equation \eqref{eqn:model} has been used to model a wide range of self-organizing phenomena in biology, physics, engineering and social sciences, such as swarming or flocking of biological organisms \cite{CarrilloVecil2010, M&K}, emergent behaviour in robotic swarms \cite{Gazi:Passino}, and opinion formation \cite{MotschTadmor2014}.  When $W$ is the Newtonian potential in $\bbr^2$ or $\bbr^3$, \eqref{eqn:model} reduces to some classical models in mathematical biology and gravitational physics, namely the Patlak-Keller-Segel model of chemotaxis \cite{KellerSegel1971} and the Smoluchowski-Poisson equation \cite{ChLaLe2004}.

The energy \eqref{eqn:energy-s} consists of an entropy (or internal energy) and an interaction energy modelled by the potential $W$. In the evolution equation \eqref{eqn:model}, these two components are reflected in the linear diffusion and the nonlocal transport term, respectively. Of particular interest is the case when $W$ is a purely {\em attractive} potential, and any two points experience a pairwise attractive interaction. In such case, there is a competing effect between the nonlocal attraction and the local diffusion, as the first leads to aggregation/blow-up and the latter results in spreading. This paper investigates the subtle balance of these opposite effects that leads to ground states of the energy functional set up on general Cartan-Hadamard manifolds.

Linear diffusion on negatively-curved manifolds of unbounded curvature is a delicate subject. The Brownian motion on Cartan-Hadamard manifolds with sectional curvatures that grow too fast at infinity, is no longer stochastically complete, i.e., the process has finite lifetime \cite[Section 15]{Grigoryan1999}. On such manifolds, the high negative curvature can sweep a Brownian particle to infinity in finite time. Stochastic incompleteness of the Brownian motion relates to the well-posedness of the Cauchy problem for the linear diffusion equation (in this case, the heat kernel no longer integrates to $1$ for all $t>0$) \cite{Grigoryan1999}. In the present work we do not address well-posedness issues for the evolution equation \eqref{eqn:model}, but only consider the free energy \eqref{eqn:energy-s} and investigate the existence of its ground states. In our study we place no restriction on the sectional curvatures, which are allowed to grow unbounded at infinity at any rate.

There is extensive literature on the free energy \eqref{eqn:energy-s} and the aggregation-diffusion equation \eqref{eqn:model} posed on $M = \bbr^\dm$. The well-posedness and the asymptotic behaviour of solutions to equation \eqref{eqn:model} on $\bbr^\dm$, along with qualitative studies on its steady states, have been investigated by numerous authors, we refer to the recent review article \cite{CarrilloCraigYao2019} for a comprehensive discussion on such issues. In \cite{CaHiVoYa2019} the authors show that the stationary solutions of \eqref{eqn:model} are radially decreasing up to translation. The equilibration toward the heat kernel was investigated in  \cite{CaGoYaZe}, and the effect of the boundaries was studied in \cite{MeFe2020}. Most relevant to the present work is the work of Carrillo {\em et al.} \cite{CaDePa2019}, where the authors establish necessary and sufficient conditions on the interaction potential and diffusion, for global energy minimizers to exist. We also note that there has been separate interest in the existence and characterization of minimizers of the energy \eqref{eqn:energy-s} without diffusion \cite{Balague_etalARMA, CaCaPa2015, ChFeTo2015, SiSlTo2015} or with nonlinear instead of linear diffusion 
\cite{BurgerDiFrancescoFranek,BuFeHu14,CaHoMaVo2018,CaHiVoYa2019,DelgadinoXukaiYao2022,Kaib17}. 

Although the free energy \eqref{eqn:energy-s} and the evolution equation \eqref{eqn:model} have been extensively studied in the Euclidean space, there is very little done for the manifold setup. In this paper we are exclusively interested in the set up of the aggregation-diffusion model on Cartan-Hadamard manifolds. Also, we assume that the interaction potential $W(x,y)$ depends only on the geodesic distance $\dist(x,y)$ between the points $x$ and $y$. This assumption follows the {\em intrinsic} approach introduced in \cite{FeZh2019}, and considered in various subsequent works on the plain interaction equation (no diffusion) on Riemannian manifolds \cite{FeHaPa2021,FePa2023a, FePa2023b,FePa2021}. As examples where the manifold framework and the intrinsic approach are particularly relevant, we mention applications to robotics and biology, where limitations in environment, the topography,  or mobility constraints restrict the agents to evolve on a certain configuration manifold.


The present work addresses and answers the following problem. Consider a Cartan-Hadamard manifold $M$ and assume that the sectional curvatures at a generic point $x$ are bounded below and above, respectively, by two negative functions $-c_m(\theta_x)$ and $-c_M(\theta_x)$, where $\theta_x$ denotes the distance from $x$ to a fixed (but arbitrary) pole on $M$. The functions $c_m(\cdot)$ and $c_M(\cdot)$ are allowed to grow at any rate at infinity, e.g., they can grow algebraically or exponentially.  To contain the diffusion on a manifold with sectional curvatures that can possibly grow unbounded at infinity, the attractive interaction potential needs to grow correspondingly at large distances. We are then interested to establish necessary and sufficient conditions on the behaviour of the interaction potential $W$ at infinity that prevents the spreading and results in existence of ground states of the energy functional. 

The problem posed above was studied in \cite{CaDePa2019} for $M = \bbr^\dm$ and in \cite{FePa2024} for Cartan-Hadamard manifolds with constant curvature bounds (the case when $c_m(\cdot)$ and $c_M(\cdot)$ are constant functions). In the Euclidean case, it was found that a sharp condition for the existence of energy minimizers is that the attractive potential grows at least {\em logarithmically} at infinity.  Similarly, for Cartan-Hadamard manifolds with sectional curvatures bounded below by a negative constant, spreading by diffusion is prevented provided $W$ grows at least {\em superlinearly}. In the present work, as we allow the sectional curvatures to grow unbounded at infinity, an even stronger growth of the attractive potential is needed to contain the diffusion. We quantify precisely such conditions on $W$ in terms of the upper and lower bounds of the sectional curvatures. 
  
To establish the results of this paper, we derive several tools that have an interest in their own. The first is a volume comparison theorem in Riemannian geometry for manifolds with non-constant bounds of the sectional curvatures. We believe that this comparison theorem is known to experts, but we could not find a reference for it. The second tool is a new logarithmic Hardy-Littlewood-Sobolev (HLS) inequality on Cartan-Hadamard manifolds of unbounded curvature. Compared to the HLS inequality on $\bbr^\dm$ \cite{CaDePa2019, CarlenLoss1992}, the logarithmic HLS inequality on Cartan-Hadamard manifolds includes an additional term that contains the Jacobian of the exponential map at a fixed pole. This additional term is a manifestation of the curvature of the manifold, as the Jacobian of the exponential map determines the volume growth of geodesic balls on the manifold. 

We also note that there has been very recent interest on the well-posedness and long-time behaviour for PDE's on Cartan-Hadamard manifolds with either linear or nonlinear diffusion and {\em local} reaction terms \cite{ GrilloMeglioliPunzo2021a, GrilloMeglioliPunzo2021b, GrilloMeglioliPunzo2023,GrilloMuratoriVazquez2017, GrilloMuratoriVazquez2019}. In particular, in \cite{GrilloMeglioliPunzo2023} the authors consider linear diffusion with a reaction term, and study the dichotomy global existence versus blow-up. Stochastic completeness of the manifold is a key assumption in their work, which limits the growth at infinity of the sectional curvatures to be at most quadratic. As pointed out above, we do not need to make such restrictive assumptions in our work, as our approach is entirely variational.

The summary of the paper is as follows. In Section \ref{sect:prelim} we present the assumptions, notations and some necessary background on comparison theorems in Riemannian geometry. In Section \ref{sec:non-existence} we establish necessary conditions on the growth of the interaction potential for ground states to exist - the main result is Theorem \ref{thm:nonexistence-gen}. In Section \ref{sec:HLS} we present a new logarithmic HLS inequality on Cartan-Hadamard manifolds with non-constant lower bound of curvatures. Section \ref{sect:nonconst} presents the proof of Theorem \ref{thm:exist-gen}, which establishes sufficient conditions on the interaction potential that guarantee existence of global energy minimizers.


\section{Assumptions and background}
\label{sect:prelim}
 
In this section we present the assumptions and set up the notations. We also present several comparison results in Riemannian geometry which will be used in our analysis. 

\subsection{Assumptions and notations}
\label{subsect:notations}
Throughout the paper we will make the following assumptions on the manifold $M$ and the interaction potential $W$.
\smallskip

\noindent(\textbf{M}) $M$ is an $\dm$-dimensional Cartan-Hadamard manifold, i.e., $M$ is complete, simply connected, and has everywhere non-positive sectional curvature. We note that as a Cartan-Hadamard manifold, $M$ has no conjugate points, it is diffeomorphic to $\bbr^\dm$ and the exponential map at any point is a global diffeomorphism. 
\smallskip

\noindent(\textbf{W}) The interaction potential $W:M\times M\to\bbr$ has the form
\[
W(x, y)=\h(\dist(x, y)),\qquad \text{ for all }x, y\in M,
\]
where $\h:[0, \infty)\to [0,\infty)$ is lower semi-continuous and non-decreasing. The fact that $h$ is non-decreasing indicates that the interaction potential is purely {\em attractive}. Also note that $h$ is bounded below, so it cannot have an infinite singularity at the origin.

We denote by $\dist$ the intrinsic distance on $M$. For $x \in M$, and $\sigma$ a two-dimensional subspace of $T_x M$, $\calK(x;\sigma)$ denotes the sectional curvature of $\sigma$ at $x$. We also denote by $\d x$ the Riemannian volume measure on $M$ and by $\calP_{ac}(M) \subset \calP(M)$ the space of probability measures on $M$ that are absolutely continuous with respect to $\d x$. We will refer to an absolutely continuous measure directly by its density $\rho$, and write $\rho \in \calP_{ac}(M)$ to mean $\d \rho (x) = \rho(x) \d x \in \calP_{ac}(M)$.

The precise definition of the free energy introduced in \eqref{eqn:energy-s} is the following. We define $E: \calP(M) \to [-\infty, \infty]$ by
\begin{equation}
\label{eqn:energy}
E[\rho]=\begin{cases}
\displaystyle \int_M \rho(x)\log\rho(x) \d x+\frac{1}{2}\iint_{M \times M}  h(\dist(x, y))\rho(x)\rho(y) \d x\d y, & \text{ if  }\rho\in \mathcal{P}_{ac}(M),\\[2pt]
\displaystyle\inf_{\substack{\rho_k \in \calP_{ac}(M) \\ \rho_k \rightharpoonup \rho}} \liminf_{k\to \infty} E[\rho_k], & \text{ otherwise},
\end{cases}
\end{equation}
where $\rho_k \rightharpoonup \rho$ denotes weak convergence as measures, i.e.,
\[
\int_M f(x) \rho_k(x) \d x \to \int_M f(x) \rho(x) \d x, \qquad \text{ as } k\to  \infty,
\]
for all bounded and continuous functions $f:M\to \bbr$.

Denote by $\calP_1(M)$ the set of probability measures on $M$ with finite first moment, i.e.,
\[
\calP_1(M) = \left \{\mu \in \calP(M) : \int_M \dist(x,x_0) \d \mu(x) < \infty \right \},
\]
for some fixed (but arbitrary) point $x_0 \in M$. For $\rho,\sigma \in \calP(M)$, the \emph{intrinsic} $1$-Wasserstein distance is defined as
\begin{equation*}
	\calW_1(\rho,\sigma) = \inf_{\gamma \in \Pi(\rho,\sigma)} \iint_{M \times M} \dist(x,y) \d\gamma(x,y),
\end{equation*}
where $\Pi(\rho,\sigma) \subset \calP(M\times M)$ is the set of transport plans between $\rho$ and $\sigma$, i.e., the set of elements in $\calP(M\times M)$ with first and second marginals $\rho$ and $\sigma$, respectively. The space $(\calP_1(M),\calW_1)$ is a metric space. 

Finally, for a fixed arbitrary point $\p \in M$ (referred throughout the paper as {\em pole}), we denote
\[
\mathcal{P}_\p(M):=\left\{\rho\in\mathcal{P}_{ac}(M) \cap \mathcal{P}_1(M): \int_M \log_\p x\rho(x)\d x=0\right\},
\]
where $\log_\p$ denotes the Riemannian logarithm map at $\p$. This is the set of admissible densities, in which we will establish existence of global energy minimizers.



\subsection{Rauch comparison theorem}
\label{subsect:Rauch}
A tool we use in our proofs is Rauch's comparison theorem \cite[Chapter 10, Proposition 2.5]{doCarmo1992}. Specifically, we use this result to compare lengths of curves on a general Cartan-Hadamard manifold with lengths of curves on spaces of constant curvature.

\begin{theorem}[Rauch comparison theorem]\label{RCT}
Let $M$ and $\tilde{M}$ be Riemannian manifolds and suppose that for all $p\in M$, $\tilde{p}\in \tilde{M}$, and $\sigma\subset T_pM$, $\tilde{\sigma}\subset T_{\tilde{p}}\tilde{M}$, the sectional curvatures $\secc$ and $\tilde{\secc}$ of $M$ and $\tilde{M}$, respectively, satisfy
\[
\tilde{\secc} (\tilde{p};\tilde{\sigma})\geq \secc(p;\sigma).
\]
Let $p\in M$, $\tilde{p}\in\tilde{M}$ and fix a linear isometry $i:T_pM\to T_{\tilde{p}}\tilde{M}$. Let $r>0$ be such that the restriction ${\exp_p}_{|B_r(0)}$ is a diffeomorphism and ${\exp_{\tilde{p}}}_{|\tilde{B}_r(0)}$ is non-singular. Let $c:[0, a]\to\exp_p(B_r(0))\subset M$ be a differentiable curve and define a curve $\tilde{c}:[0, a]\to\exp_{\tilde{p}}(\tilde{B}_r(0))\subset\tilde{M}$ by
\[
\tilde{c}(s)=\exp_{\tilde{p}}\circ i\circ\exp^{-1}_p(c(s)),\qquad s\in[0, a].
\]  
Then the length of $c$ is greater or equal than the length of $\tilde{c}$.
\end{theorem}


\subsection{Bounds on the Jacobian of the exponential map.}
\label{subsect:bounds-J}
We will follow notations and the general framework from \cite{Chavel2006}.  A Jacobi field along a unit-speed geodesic $\gamma$ is a differentiable vector field $Y$ along $\gamma$ that satisfies the Jacobi equation
\begin{equation}
\label{eqn:Jacobi}
\nabla_t^2 Y + \calR(\gamma',Y) \gamma' = 0,
\end{equation}
where $\nabla$ denotes the Levi-Civita connection on $M$ and $\calR$ represents the curvature tensor. Denote by $\calJ$ the vector space of Jacobi fields along $\gamma$ and by $\calJ^\perp$ the subspace of $\calJ$ consisting of the Jacobi fields orthogonal to $\gamma$. Jacobi fields are variation fields through geodesics, and they measure the spread and the volume growth of a geodesic spray.

Consider a point $x \in M$, a unit tangent vector $u \in T_x M$, and a unit-speed geodesic $\gamma$ through $x$ in the direction $u$, i.e., $\gamma(0) = x$, $\gamma'(0) = u$. Set an orthonormal basis $\{e_1, \cdots, e_\dm\}$ of  $T_x M$, with $e_\dm = u$, and consider the Jacobi fields $J_j(t)$, $j=1,\dots, \dm-1$, along the geodesic $\gamma$ that satisfy
\[
J_j(0)=0\qquad\text{and}\qquad J_j'(0)=e_j, \qquad  j =1,\dots,\dm-1.
\]
We point out that $J_j(t) \in \calJ^\perp$ for all $t \geq 0$, $j=1,\dots, \dm-1$. Denote by $\calA(t;u)$ the matrix with columns $J_j(t)$, $j = 1,\dots,\dm-1$, i.e., 
\[
\mathcal{A}(t)=[J_1(t) \;  J_2(t) \; \dots \; J_{\dm-1}(t)].
\]
Note that the matrix $\mathcal{A}(t)$ depends on the geodesic $\gamma$, in particular on the unit tangent vector $u$. When this dependence is important and can lead to confusion, we will use the notation $\mathcal{A}(t;u)$ instead.

By the properties of the Jacobi fields, we have
\[
J_j(t)=(d\exp_x)_{te_\dm}(te_j),\qquad j =1,\dots, \dm-1,
\]
and hence we can write
\begin{equation}
\label{eqn:calA-exp}
\mathcal{A}(t)=t[(d\exp_x)_{te_\dm}(e_1) \quad (d\exp_x)_{te_\dm}(e_2) \quad  \dots \quad  (d\exp_x)_{te_\dm}(e_{\dm-1})].
\end{equation}

In the present work we are particularly interested in the Jacobian $J(\exp_x)$ of the exponential map.  Since
\[
(d\exp_x)_{te_\dm}(e_\dm)=e_\dm,
\]
by using \eqref{eqn:calA-exp} we can express
\begin{equation}
\label{eqn:Jexp-A}
J(\exp_x)(te_\dm)=\det( (d\exp_x)_{te_\dm})=\det\begin{bmatrix}
\frac{\mathcal{A}(t)}{t}&0\\[5pt]
0&1
\end{bmatrix}=\det(\mathcal{A}(t)/t)=\frac{\det\mathcal{A}(t)}{t^{\dm-1}}.
\end{equation}

Throughout the paper we fix a generic point $\p$ in $M$ which we will be referring to as the {\em pole} of the manifold. In general, a point on a manifold is called a pole if the exponential map at that point is a global diffeomorphism. On Cartan-Hadamard manifolds all points satisfy this property. We also make the notation
\[
\theta_x := \dist(o,x), \qquad \text{ for all } x \in M.
\]

The following theorem establishes estimates on the Jacobian $J(\exp_x)$. The theorem is a generalization of \cite[Theorems III.4.1 and III.4.3]{Chavel2006}, which consider the case when the sectional curvatures of $M$ are bounded above and below by constants.

\begin{theorem}(Generalizations of \cite[Theorems III.4.1 and III.4.3]{Chavel2006})
\label{lemma:Chavel-thms-gen}
Suppose the sectional curvatures of  $M$ satisfy
\begin{equation}
-c_m(\theta_x)\leq \mathcal{K}(x;\sigma)\leq -c_M(\theta_x) < 0, 
\label{eqn:var-bounds}
\end{equation}
for all $x \in M$ and all two-dimensional subspaces $\sigma \subset T_x M$, where $c_m(\cdot)$ and $c_M(\cdot)$ are positive continuous functions of the distance from the pole $\p$. Then, for any $x \in M$ and a unit-speed geodesic $\gamma$ through $x$, it holds that 
\[
\psi_M^{\dm-1}(t) \leq \det\mathcal{A}(t)\leq \psi_m^{\dm-1}(t),
\]
where $\psi_m$ and $\psi_M$ are solutions of
\begin{equation}
\begin{cases}
\psi_m''(\theta)=c_m(\theta)\psi_m(\theta),\quad\forall \theta>0, \\[5pt]
\psi_m(0)=0,\quad \psi_m'(0)=1,
\end{cases}\quad \text{ and } \quad 
\begin{cases}
\psi_M''(\theta)=c_M(\theta)\psi_M(\theta),\quad\forall \theta>0, \\[5pt]
\psi_M(0)=0,\quad \psi_M'(0)=1.
\end{cases}
\label{eqn:psimM}
\end{equation}
In addition (by \eqref{eqn:Jexp-A}), given that $\gamma$ is arbitrary,
\[
\left( \frac{\psi_M(\|u\|)}{\|u\|} \right)^{\dm-1}\leq |J(\exp_x)(u)|\leq
\left( \frac{\psi_m(\|u\|)}{\|u\|} \right)^{\dm-1},
\]
for all $u \in T_x M$.
\end{theorem}
\begin{proof} The approach follows the proofs of \cite[Theorems III.4.1 and III.4.3]{Chavel2006}, with some small caveats related to the nonconstant bounds of curvature. For completeness, we will present the proof in detail in Appendix \ref{appendix:Chavel-gen}. \end{proof}

\begin{remark}
\label{rmk:psimM} 
\normalfont
While in general one cannot find explicit solutions to the IVP's in \eqref{eqn:psimM}, these are linear problems and the global well-posedness of their solutions is guaranteed by standard results in ODE theory.
In addition, it also holds that
\begin{equation*}
 \psi_m(\theta) \geq 0, \; \psi_M(\theta) \geq 0 \qquad \text{ and } \qquad \psi_m'(\theta)\geq 1, \; \psi_M'(\theta)\geq 1, \quad \text{ for all } \theta \geq 0.
\end{equation*}
To show this claim, one can argue as following. Since $\psi_m(0)=0$ and $\psi_m'(0)=1$, there exists a small $\delta >0 $ such that $\psi_m(\theta) \geq 0$ for all $\theta\in [0, \delta]$. Also, since $c_m(\cdot) \geq 0$, we also have $\psi_m''(\theta)=c(\theta)\psi_m(\theta)\geq0$ for all $\theta\in[0, \delta]$, which implies that $\psi_m'$ is non-decreasing, and hence, $\psi_m'(\theta)\geq1$ for all $\theta\in[0, \delta]$. Let $T=\sup_{t \geq 0}\{t: \psi_m(\theta) \geq 0, \, \psi_m'(\theta)\geq1 \text{ for all } 0\leq \theta\leq t\}$. If $T<\infty$, then we can apply a similar argument as above to get $\psi_m(\theta) \geq 0$, $\psi_m'(\theta)\geq1$ for all $[T, T+\delta]$, for some $\delta>0$. So, we can conclude that $T=\infty$. A similar argument works for $\psi_M$. 
\end{remark}

\begin{remark} 
\label{rmk:c-consts}
\normalfont
When $c_m(\cdot) \equiv c_m$ and $c_M(\cdot) \equiv c_M$ are constant functions, the IVP's in \eqref{eqn:psimM} have the explicit solutions
 \[
\psi_m(\theta)=\frac{\sinh(\sqrt{c_m}\theta)}{\sqrt{c_m}},\quad \text{ and } \quad \psi_M(\theta)=\frac{\sinh(\sqrt{c_M}\theta)}{\sqrt{c_M}}.
\]
In this case, Theorem \ref{lemma:Chavel-thms-gen} reduces to the comparison results from \cite[Theorems III.4.1 and III.4.3]{Chavel2006}, i.e., for any $x \in M$ and a unit-speed geodesic $\gamma$ through $x$, it holds that 
\[
\left(\frac{\sinh(\sqrt{c_M}t)}{\sqrt{c_M}}\right)^{\dm-1}\leq \det\mathcal{A}(t)\leq \left(\frac{\sinh(\sqrt{c_m}t)}{\sqrt{c_m}}\right)^{\dm-1},
\]
and
\[
\left(
\frac{\sinh(\sqrt{c_M}\|u\|)}{\sqrt{c_M}\|u\|}
\right)^{\dm-1}\leq |J(\exp_x)(u)|\leq\left(
\frac{\sinh(\sqrt{c_m}\|u\|)}{\sqrt{c_m}\|u\|}
\right)^{\dm-1},
\]
for all $u \in T_x M$.
\end{remark}

\begin{remark}
\normalfont
We note that for the upper bound on $\det\mathcal{A}(t)$ a weaker condition on curvature can be assumed, namely that the Ricci curvature along $\gamma$ is greater than or equal to $-(\dm-1) c_m$ (see Theorem III.4.3 in \cite{Chavel2006}). We do not consider this weaker assumption here.
\end{remark}

Theorem \ref{lemma:Chavel-thms-gen} can be used to find bounds on the volume of geodesic balls in $M$. In the paper, $\bbs^{\dm -1}$ denotes the unit sphere in $\bbr^\dm$ and $\omega(\dm)$ the volume of the unit ball in $\bbr^\dm$.  Also, $B_\theta(\p)$ denotes the geodesic ball in $M$ centred at $\p$ of radius $\theta$, and $|B_\theta(\p)|$ denotes its volume.
\begin{corollary}(Generalizations of \cite[Theorems III.4.2 and III.4.4]{Chavel2006})
\label{cor:AV-bounds}
Suppose the sectional curvatures of $M$ satisfy \eqref{eqn:var-bounds}. Then, 
\[
 \dm w(\dm)\int_0^\theta \psi_M(t) ^{\dm-1} \d t \leq |B_{\theta}(\p)|\leq \dm w(\dm) \int_0 ^\theta \psi_m(t)^{\dm-1} \d t,
\]
with $\psi_m$ and $\psi_M$ defined as in Theorem \ref{lemma:Chavel-thms-gen}. 
\end{corollary}
\begin{proof} 
We will use here the full notation  $\calA(t; u)$ to indicate the dependence of $\calA$ on the unit tangent vector $u = \gamma'(0)$.  By using geodesic spherical coordinates centred at $\p$, the Riemannian measure on $M$ is given by (see \cite[Theorem III.3.1]{Chavel2006})
\begin{equation}
\label{eqn:vol-form}
\d x(\exp_\p(t u)) = \det\mathcal{A}(t;u) \d t \d \sigma(u),
\end{equation}
where $\d \sigma(u)$ denotes the Riemannian measure on $\bbs^{\dm -1}$.  The proof of the corollary now follows immediately from Theorem \ref{lemma:Chavel-thms-gen} together with \eqref{eqn:vol-form}, by writing
\[
| B_\theta(\p)| = \int_0^\theta \int_{\bbs^{\dm-1}} \det\mathcal{A}(t;u) \d \sigma(u) \d t.
\]
\end{proof}

\begin{remark}
\normalfont
When $c_m(\cdot) \equiv c_m$ and $c_M(\cdot) \equiv c_M$, Corollary \ref{cor:AV-bounds} reduces to
\[
 \dm w(\dm)\int_0^\theta \left(\frac{\sinh(\sqrt{c_M}t)}{\sqrt{c_M}}\right)^{\dm-1} \d t \leq |B_{\theta}(\p)|\leq \dm w(\dm) \int_0 ^\theta \left(\frac{\sinh(\sqrt{c_m} t)}{\sqrt{c_m}}\right)^{\dm-1} \d t,
\]
which is the result in \cite[Theorems III.4.2 and III.4.4]{Chavel2006}. The two bounds in this case represent the volume of the ball of radius $\theta$ in the hyperbolic space $\bbh^\dm$ of constant curvatures $-c_M$, and $-c_m$, respectively.
\end{remark}


\subsection{Properties of solutions to the IVP \eqref{eqn:psimM}.}

Next, we provide various properties of solutions to the IVP's in \eqref{eqn:psimM}, under additional assumptions on the functions $c_m$ and $c_M$. We list the results for a generic function $c$.  Throughout the paper, $\underline{D}f$ and $\overline{D}f$ denote, respectively, the lower and the upper derivatives of a real-valued function $f$, defined by
\[
\underline{D}f(x)=\liminf_{h\to0}\frac{f(x+h)-f(x)}{h},\qquad \overline{D}f(x)=\limsup_{h\to0}\frac{f(x+h)-f(x)}{h}.
\]

\begin{lemma}\label{estpsi}
Let $c:[0,\infty) \to (0,\infty)$ be a positive, continuous, and non-decreasing function. Consider the solution $\psi$ of the IVP:
\begin{equation}
\label{eqn:ivp-psi}
\begin{cases}
\psi''(\theta)=c (\theta)\psi (\theta),\qquad\forall \theta>0, \\[5pt]
\psi(0)=0,\quad \psi'(0)=1.
\end{cases}
\end{equation}
Then, for any $\epsilon>0$, there exists $\theta_0>0$ such that
\begin{align}\label{estpsires-relaxed}
\psi(\theta)\leq
\psi(\theta_0)\exp\left(
(1+\epsilon)\int_{\theta_0}^\theta\sqrt{c(t)}\, \d t \right), \qquad \forall  \theta \geq \theta_0.
\end{align}

Furthermore, assume in addition that $c(\cdot)$ satisfies
\begin{equation}
\label{eqn:c32}
\lim_{\theta\to\infty}\frac{\overline{D}c(\theta)}{c(\theta)^{3/2}}=0 \qquad\text{ and }\qquad \underline{D}c(\theta)>0, \qquad \forall \theta\geq0.
\end{equation}
Then, for any $0<\epsilon <1$, there exists $\theta_0>0$ such that
\begin{align}\label{estpsires}
\psi(\theta_0)\exp\left(
(1-\epsilon)\int_{\theta_0}^\theta\sqrt{c(t)}\, \d t
\right)\leq \psi(\theta), \qquad \forall \theta \geq \theta_0.
\end{align}
\end{lemma}
\begin{proof}
See Appendix \ref{appendix:estpsi}.
\end{proof} 

\begin{remark}
\label{rmk:c-ex}
\normalfont
There are various examples of functions $c(\cdot)$, relevant to our work, which satisfy condition \eqref{eqn:c32}, e.g., 
\[
c(\theta) = \theta^k,\quad c(\theta)= e^{\theta^k},\quad c(\theta) = 1-e^{-\theta^k},
\]
for all $k>0$. Specifically, the first two examples above, when used as bounds for the curvature (see \eqref{eqn:var-bounds}), correspond to manifolds whose negative curvature grows algebraically, respectively exponentially, at infinity.
\end{remark}

\begin{remark}
\label{rmk:3/2cond}
\normalfont
To give some intuition on the first assumption in \eqref{eqn:c32}, consider a non-decreasing $C^1$ function $c(\cdot)$ which satisfies
\[
\lim_{\theta\to\infty}\frac{c'(\theta)}{c(\theta)^{3/2}}>0.
\]
Then, there exists $\delta>0$ and $\theta_0>0$ such that
\[
\frac{c'(\theta)}{c(\theta)^{3/2}}\geq\delta, \qquad  \text{ for all }\theta\geq\theta_0.
\]
By solving the differential inequality, we get
\[
-\frac{2}{c(\theta)^{1/2}}+\frac{2}{c(\theta_0)^{1/2}}\geq \delta (\theta-\theta_0), 
\]
which implies
\[
c(\theta)\geq \left( \frac{1}{c(\theta_0)^{1/2}}-\frac{\delta(\theta-\theta_0)}{2} \right)^{-2}.
\]
From here one infers that $c(\theta) \to \infty$ as $\theta \to \theta_0+\frac{2}{\delta c(\theta_0)^{1/2}}$ and hence, $c(\theta)$ is not defined globally for all $\theta \geq 0$. Given these considerations, assumption \eqref{eqn:c32} is natural for a globally defined function $c(\cdot)$, as used in this paper.
\end{remark}

 \begin{lemma}\label{lem:linearlog}
Let $c:[0,\infty) \to (0,\infty)$ be a positive, continuous, and non-decreasing function, and $\psi$ be the solution of the IVP \eqref{eqn:ivp-psi}. Then, we have
\[
\liminf_{\theta\to\infty}\frac{\log\left(\frac{\psi(\theta)}{\theta}\right)}{\theta}\geq A,
\] 
for some $A>0$.
\end{lemma}
\begin{proof}
Since $c(\cdot)$ is non-decreasing, we have
\[
\psi''(\theta)=c(\theta)\psi(\theta)\geq c(0) \psi(\theta), \qquad \forall \theta\geq 0.
\]
Denote by $\psi_0(\theta)$ the solution of
\[
\psi_0''(\theta)=c(0) \psi_0(\theta),\quad \psi_0(0)=0,\quad \psi_0'(0)=1,
\]
which is given explicitly by
\[
\psi_0(\theta)=\frac{\sinh(\sqrt{c(0)}\theta)}{\sqrt{c(0)}}.
\]

It holds that $\psi(\theta)\geq \psi_0(\theta)$, for all $\theta \geq 0$, and hence,
\begin{align*}
\frac{\log\left(\frac{\psi(\theta)}{\theta}\right)}{\theta} 
& \geq \frac{\log\left(\frac{\sinh(\sqrt{c(0} \, \theta)}{ \sqrt{c(0)} \theta}\right)}{\theta} \\
&=\sqrt{c(0)} +\frac{\log\left(\frac{1-e^{-2 \sqrt{c(0)} \theta}}{2 \sqrt{c(0)} \theta}\right)}{\theta}.
\end{align*}
From L'H\^{o}pital theorem, we get
\begin{align*}
\lim_{\theta\to\infty}\frac{\log\left(\frac{1-e^{-2\sqrt{c(0)} \theta}}{2 \sqrt{c(0)} \theta}\right)}{\theta}& =\lim_{\theta\to\infty}\frac{1}{\frac{1-e^{-2 \sqrt{c(0)} \theta}}{2 \sqrt{c(0)} \theta}} \cdot  \frac{2 \sqrt{c(0)} e^{-2\sqrt{c(0)} \theta} \theta - (1-e^{-2\sqrt{c(0)} \theta})}{2 \sqrt{c(0)}\theta^2} \\
&=\lim_{\theta\to\infty}\frac{e^{-2\sqrt{c(0)}\theta}(2 \sqrt{c(0)} \theta +1) -1}{\theta(1-e^{-2\sqrt{c(0)} \theta})} \\[2pt]
&=0,
\end{align*}
which implies
\[
\liminf_{\theta\to\infty}\frac{\log\left(\frac{\psi(\theta)}{\theta}\right)}{\theta}\geq \sqrt{c(0)}.
\]
We conclude the statement of the lemma with $A = \sqrt{c(0)}>0$.
\end{proof}


\section{Nonexistence of a global minimizer }
\label{sec:non-existence}

We will show that the energy functional \eqref{eqn:energy} does not admit global minimizers when the attractive forces at infinity are not strong enough to contain the diffusion. We fix a pole $\p \in M$ and assume that the sectional curvatures $\mathcal{K}(x;\sigma)$ are bounded above by a function $-c_M(\theta_x)<0$ that can grow unbounded as $\theta_x \to \infty$. We denote by $\psi_M$  the solution to the initial-value problem
\begin{equation}
\label{eqn:IVP-psiM}
\begin{cases}
\psi_M''(\theta)=c_M(\theta)\psi_M(\theta), \qquad \forall \theta>0,\\[5pt]
\psi_M(0)=0,\quad \psi_M'(0)=1.
\end{cases}
\end{equation}

The main result is the following.
\begin{theorem}[Nonexistence by spreading]
\label{thm:nonexistence-gen}
Let $M$ be an $n$-dimensional Cartan-Hadamard manifold and $\p \in M$ a fixed pole. Assume that the sectional curvatures of $M$ satisfy
\[
\mathcal{K}(x; \sigma)\leq -c_M(\theta_x),
\]
for all $x \in M$ and all two-dimensional subspaces $\sigma \subset T_x M$, where
$c_M(\cdot)$ is a positive and continuous function such that
\begin{equation}
\label{eqn:c32-cM}
\lim_{\theta\to\infty}\frac{\overline{D}c_M(\theta)}{c(\theta)^{3/2}}=0\qquad\text{and}\qquad \underline{D}c_M(\theta)>0,\qquad\forall \theta\geq0.
\end{equation}
Let $h:[0,\infty) \to [0,\infty)$ be a non-decreasing lower semi-continuous function that satisfies
\begin{equation}\label{nonexcondi2}
\limsup_{\theta \to\infty}\left(h(\theta )-A \int_0^{\theta-\delta}\sqrt{c_M \left(t/2\right)}\, \d t \right)=-\infty,
\end{equation}
for some $A<\dm-1$, and some $\delta>0$.  Then, the energy functional \eqref{eqn:energy} admits no global minimizer in $\calP_{ac}(M)$.
\end{theorem}
\begin{proof}
Consider the following family of densities $\rho_R$ defined for all $R>0$:
\begin{equation*}
\rho_{R}(x)=\begin{cases}
\displaystyle\frac{1}{|B_R(\p)|},&\quad\text{when }x\in B_R(\p),\\[10pt]
0,&\quad\text{otherwise}.
\end{cases}
\end{equation*}
Then, we have
\begin{align}
E[\rho_R]
&= \int_M \rho_{R}(x)\log\rho_{R}(x)\d x+\frac{1}{2}\iint_{M \times M} h(\dist(x, y))\rho_{R}(x)\rho_{R}(y)\d x \d y \nonumber \\
&\leq-\log|B_R(\p)|+\frac{h(2R)}{2}, \label{eqn:E-rhoR}
\end{align}
where for the inequality we used 
\[
\sup_{x, y\in\mathrm{supp}(\rho_R)}\dist(x, y)=2R.
\]
We will show that $E[\rho_R] \to - \infty$ as $R \to \infty$, and hence, no global energy minimizers in $\calP_{ac}(M)$ exist. 

Since $A<\dm-1$, set $\epsilon \in (0,1)$ by 
\begin{equation}
\label{eqn:A}
A = (1-\epsilon)(\dm-1).
\end{equation}
By the assumptions on $c_M$ and Lemma \ref{estpsi}, there exists $\theta_0>0$ such that
\begin{equation}
\label{eqn:psiM-lb}
\psi_M(\theta)\geq \psi_M(\theta_0)\exp\left( (1-\epsilon)\int_{\theta_0}^\theta\sqrt{c_M(t)}\, \d t \right), \qquad \text{ for all } \theta > \theta_0,
\end{equation}
where $\psi_M$ is the solution of the IVP \eqref{eqn:IVP-psiM}.

Then, for any $R>\theta_0+\frac{\delta}{2}$ (with $\delta$ as in \eqref{nonexcondi2}), combine Corollary \ref{cor:AV-bounds} part ii)  and \eqref{eqn:psiM-lb} to find
\begin{align*}
|B_R(\p)|& \geq  \dm w(\dm)\int_0^R \psi_M(\theta) ^{\dm-1} \d \theta \\
&\geq   \dm w(\dm)\int_{R - \frac{\delta}{2}}^R \psi_M(\theta) ^{\dm-1} \d \theta  \\
&\geq \dm w(\dm) \int_{R - \frac{\delta}{2}}^R\psi_M(\theta_0)^{\dm-1}\exp\left(
(1-\epsilon)(\dm-1)\int_{\theta_0}^\theta \sqrt{c_M(t)} \, \d t \right)\d \theta.
\end{align*}
Furthermore, using that $\displaystyle \int_{\theta_0}^\theta \sqrt{c_M(t)}\, \d t \geq  \int_{\theta_0}^{R - \frac{\delta}{2}} \sqrt{c_M(t)}\, \d t $, for all $R - \frac{\delta}{2} < \theta < R$, we get
\begin{align*}
|B_R(\p)| &\geq\dm w(\dm)\psi_M(\theta_0)^{\dm-1} \int_{R-\frac{\delta}{2}}^R\exp\left(
(1-\epsilon)(\dm-1)\int_{\theta_0}^{R-\frac{\delta}{2}} \sqrt{c_M(t)} \, \d t \right) \d\theta\\
&=\frac{\delta}{2} \dm w(\dm)\psi_M(\theta_0)^{\dm-1} \exp\left(
(1-\epsilon)(\dm-1)\int_{\theta_0}^{R-\frac{\delta}{2}} \sqrt{c_M(t)} \, \d t
\right).
\end{align*}
The above yields
\[
\log|B_R(\p)|\geq \log \left(\frac{\delta}{2} \dm w(\dm)\psi_M(\theta_0)^{\dm-1} \right)+(1-\epsilon)(\dm-1)\int_{\theta_0}^{R-\frac{\delta}{2}} \sqrt{c_M(t)}\, \d t.
\]

We can now estimate $E[\rho_R]$ for $R> \theta_0 + \frac{\delta}{2}$ (see \eqref{eqn:E-rhoR}) as follows:
\begin{align*}
E[\rho_R] &\leq -\log \left(\frac{\delta}{2} \dm w(\dm)\psi_M(\theta_0)^{\dm-1} \right)-(1-\epsilon)(\dm-1)\int_{\theta_0}^{R-\frac{\delta}{2}} \sqrt{c_M(t)} \, \d t+\frac{h(2R)}{2}\\[2pt]
&=-\log \left(\frac{\delta}{2} \dm w(\dm)\psi_M(\theta_0)^{\dm-1} \right)+(1-\epsilon)(\dm-1)\int_{0}^{\theta_0} \sqrt{c_M(t)}\, \d t \\
& \quad -(1-\epsilon)(\dm-1)\int_{0}^{R-\frac{\delta}{2}} \sqrt{c_M(t)} \,\d t+\frac{h(2R)}{2}.
\end{align*}
Since by the assumption \eqref{nonexcondi2} on $h$ (see also \eqref{eqn:A}) and a simple change of variable, we have
\begin{equation*}
\limsup_{R\to\infty}\left(\frac{h(2R)}{2}-(1-\epsilon)(\dm-1)\int_0^{R-\frac{\delta}{2}}\sqrt{c_M(t)}\, \d t\right)=-\infty,
\end{equation*}
we infer $\lim_{R\to\infty}E[\rho_R]=-\infty$. 
\end{proof}

\begin{example}
\label{ex:nonexist}
\normalfont
For certain functions $c_M$ we can compute $ \int_0^{\theta-\delta}\sqrt{c_M \left(t/2\right)}\, \d t $ and hence, find explicit expressions for the behaviour of $h$ at infinity that lead to spreading.
\begin{enumerate}
\item Constant: $c_M(\theta) \equiv c_M$. This is the case studied in \cite{FePa2024}. Note however that constant functions do not satisfy the second condition in \eqref{eqn:c32-cM}, so Theorem \ref{thm:nonexistence-gen} cannot be applied directly. Nevertheless, compute (see \eqref{nonexcondi2}):
\[
h(\theta)- A \int_{0}^{\theta-\delta}\sqrt{c_M}\, \d t=h(\theta)-A \sqrt{c_M}\theta+A \delta \sqrt{c_M}.
\]
Then, had the statement of Theorem \ref{thm:nonexistence-gen} applied to this case, we recover the result in \cite[Theorem 2.1]{FePa2024}.  

\item Power law: $c_M(\theta)=\theta^k$ with $k\geq1$. We compute
\[
h(\theta)-A \int_{0}^{\theta-\delta}2^{-k/2} t^{k/2}\d t=h(\theta)-\frac{A 2^{-k/2}} {k/2+1}\left(\theta-\delta\right)^{k/2+1},
\]
which implies that spreading occurs if $h$ grows slower than $\theta^{k/2+1}$ at infinity.

\item Exponential growth: $c_M(\theta)= e^{\beta\theta}$ with $\beta >0$. We find
\[
h(\theta)-A\int_{0}^{\theta-\delta} e^{\beta t/4} \d t=h(\theta)-
\frac{A}{{\beta/4}}\left(e^{\beta(\theta-\delta)/4}-1\right).
\]
Hence, in this case spreading cannot be prevented provided $h$ grows slower than $e^{\beta \theta /4}$ at infinity.
\end{enumerate}
\end{example}


\section{Logarithmic HLS inequality on Cartan-Hadamard manifolds}\label{sec:HLS}
This section presents a new logarithmic HLS inequality on Cartan-Hadamard manifolds, which is a key tool for our results on existence of global energy minimizers. We first recall the logarithmic HLS inequality on $\bbr^\dm$, given by the following theorem.
\begin{theorem}[Logarithmic HLS inequality on $\bbr^\dm$ \cite{CaDePa2019}]
\label{thm:HLS-Rd}
Let $\rho\in\mathcal{P}_{ac}(\bbr^\dm)$ satisfy $\log(1+|\cdot|^2)\rho\in L^1(\bbr^\dm)$. Then there exists $C_0\in \bbr$ depending only on $\dm$, such that
\[
-\int_{\bbr^\dm}\int_{\bbr^\dm}\log(|x-y|)\rho(x)\rho(y)\d x \d y \leq \frac{1}{\dm}\int_{\bbr^\dm} \rho(x)\log\rho(x) \d x+C_0.
\]
\end{theorem}

We generalize Theorem \ref{thm:HLS-Rd} to Cartan-Hadamard manifolds.  The result is the following. 

\begin{theorem}[Logarithmic HLS inequality on Cartan-Hadamard manifolds] 
\label{HLSub}
Let $M$ be an $\dm$-dimensional Cartan-Hadamard manifold, and $\p \in M$ be a fixed (arbitrary) pole. Assume that the sectional curvatures of $M$ satisfy, for all $x \in M$ and all sections $\sigma \subset T_x M$, 
\begin{equation}
\label{eqn:K-lowerB}
-c_m(\theta_x)\leq \mathcal{K}(x;\sigma)\leq 0,
\end{equation}
where $c_m(\cdot)$ is a continuous positive function.  Let $\psi_m$ be the solution to the following IVP:
\begin{equation}
\label{eqn:IVP-psim}
\begin{cases}
\psi_m''(\theta)=c_m(\theta) \psi_m(\theta), \qquad \forall \theta>0,\\[5pt]
\psi_m(0)=0,\quad \psi'_m(0)=1.
\end{cases}
\end{equation}
Then, if $\rho\in \mathcal{P}_{ac}(M)$ and $\displaystyle  \int_M \log(1+|\theta_x|^2)\rho(x)\d x<\infty$, we have
\begin{equation}
\label{eqn:HLS-CH-II}
\begin{aligned}
&-\int_M\int_M\log(\dist(x, y))\rho(x)\rho(y)\d x \d y\leq \\
& \hspace{3.5cm}  \frac{1}{\dm}\int_M \rho(x)\log\rho(x)\d x+\frac{(\dm-1)}{\dm}\int_M \log\left ( \frac{\psi_m(\theta_x)}{\theta_x}\right ) \rho(x) \d x+C_0,
\end{aligned}
\end{equation}
where $C_0\in\bbr$ is the constant depending only on $\dm$ introduced in Theorem \ref{thm:HLS-Rd}.
\end{theorem}
\begin{proof}
Denote by $f:M\to T_\p M$ the Riemannian logarithm map at $\p$, i.e.,
\begin{equation}
\label{eqn:Rlog}
f(x)=\log_\p x,\qquad \text{ for all }x\in M.
\end{equation}
The inverse map $f^{-1}:T_\p M\to M$ is the Riemannian exponential map
\[
f^{-1}(u)=\exp_\p u, \qquad \text{ for all } u\in T_\p M.
\]
On Cartan-Hadamard manifolds, the exponential map is a global diffeomorphism.

Take a density $\rho\in \mathcal{P}_{ac}(M)$, and its pushforward (as measures) $f_\#\rho\in \mathcal{P}_{ac}(T_\p M)$ by $f$. Then,
\begin{equation}
\label{eqn:push-rho}
f_\#\rho(u)=\rho(f^{-1}(u))|J(f^{-1})(u)|, \qquad \text{ for all } u\in T_\p M.
\end{equation}
Since $M$ has non-positive curvature, by Rauch Comparison Theorem (see Theorem \ref{RCT}) we have
\begin{equation}
\label{eqn:dist-ineq}
\dist(x, y)\geq |f(x)-f(y)|, \qquad \text{ for all } x,y \in M.
\end{equation}
Inequality \eqref{eqn:dist-ineq} can be justified as follows. Fix any two points $x,y \in M$. In the setup of Theorem \ref{RCT}, take $\tilde{M} = T_\p M$, $i$ to be the identity map, and $c$ to be the geodesic curve joining $x$ and $y$. The assumption on the curvatures $\secc$ and $\tilde{\secc}$ holds, as $\secc \leq 0 = \tilde{\secc}$. The curve $\tilde{c}$ constructed in Theorem \ref{RCT} joins $\log_\p x$ and $\log_\p y$, and hence, its length is greater than or equal to the length of the straight segment joining $\log_\p x$ and $\log_\p y$. Together with Theorem \ref{RCT}, we then infer:
\[
|\log_\p x - \log_\p y| \leq l(\tilde{c}) \leq l(c) = \dist(x,y).
\]
Note that \eqref{eqn:dist-ineq} holds indeed with equal sign for $M=\bbr^\dm$. 

By \eqref{eqn:dist-ineq}, we have
\begin{equation}
\label{eqn:HLS-deriv1}
-\int_M\int_M\log(\dist(x, y))\rho(x)\rho(y)\d x \d y\leq -\int_M\int_M\log(|f(x)-f(y)|)\rho(x)\rho(y)\d x \d y,
\end{equation}
since $\log(\cdot)$ is an increasing function on $(0,\infty)$. Then, using the definition of the pushforward measure and the logarithmic HLS inequality on $T_\p M \simeq \bbr^\dm$, we get 
\begin{align}
\label{eqn:HLS-deriv2}
 -\int_M\int_M\log(|f(x)-f(y)|)\rho(x)\rho(y)\d x \d y &= -\int_{T_\p M}\int_{T_\p M}\log(|w-z|)f_\#\rho(w)f_\#\rho(z)\d w \d z \nonumber \\
 &\leq \frac{1}{\dm}\int_{T_\p M}f_\#\rho(z)\log(f_\#\rho(z))\d z+C_0,
\end{align}
where $C_0$ is a constant that depends on $\dm$. Now, use again the property of the pushforward measure, to find
\begin{align}
\frac{1}{\dm}\int_{T_\p M}f_\#\rho(z)\log(f_\#\rho(z))\d z& +C_0=\frac{1}{\dm}\int_M \rho(x)\log(f_\#\rho(f(x))) \d x+C_0 \nonumber \\
&=\frac{1}{\dm}\int_M \rho(x)\log(\rho(x)|J(f^{-1})(f(x))| )\d x+C_0 \nonumber \\
&=\frac{1}{\dm}\int_M \rho(x)\log\rho(x)\d x+\frac{1}{\dm}\int_M \rho(x)\log|J(f^{-1})(f(x))| \d x+C_0.
\label{eqn:HLS-deriv3}
\end{align}

By combining \eqref{eqn:HLS-deriv1}, \eqref{eqn:HLS-deriv2} and \eqref{eqn:HLS-deriv3}, we obtain
\begin{align}
\label{eqn:HLS-CH}
\begin{aligned}
& -\int_M\int_M\log(\dist(x, y))\rho(x)\rho(y)\d x \d y  \leq \\
& \hspace{3.5cm} \frac{1}{\dm}\int_M \rho(x)\log\rho(x)\d x+\frac{1}{\dm}\int_M \rho(x)\log|J(f^{-1})(f(x))| \d x+C_0.
 \end{aligned}
\end{align}

Compared to the logarithmic HLS inequality on $\bbr^\dm$, inequality \eqref{eqn:HLS-CH} has an additional term on the right-hand-side, containing the Jacobian of the exponential map. Using the lower bound \eqref{eqn:K-lowerB} of the sectional curvatures we can bound this Jacobian from above by
\[
|J(f^{-1})(f(x))|\leq  \left(
\frac{\psi_m(\|f(x)\|)}{\|f(x)\|}\right)^{\dm-1} = \left(
\frac{\psi_m(\theta_x)}{\theta_x}
\right)^{\dm-1}.
\]

The HLS inequality \eqref{eqn:HLS-CH-II} now follows from the estimate above and \eqref{eqn:HLS-CH}.
\end{proof}

\begin{remark}
In particular, if $c_m$ is a constant function,  $ \displaystyle \psi_m(\theta)=\sinh(\sqrt{c_m}\theta)/\sqrt{c_m}$, and \eqref{eqn:HLS-CH-II} reduces to
\begin{equation*}
\begin{aligned}
&-\int_M\int_M\log(\dist(x, y))\rho(x)\rho(y)\d x \d y\leq \\
& \hspace{3.5cm}  \frac{1}{\dm}\int_M \rho(x)\log\rho(x)\d x+\frac{(\dm-1)}{\dm}\int_M \log\left ( \frac{\sinh(\sqrt{c_m}\theta_x)}{\sqrt{c_m}\theta_x} \right ) \rho(x) \d x+C_0.
\end{aligned}
\end{equation*}
This recovers the result in \cite[Theorem 5.2]{FePa2024}.
\end{remark}


\section{Existence of a global minimizer}
\label{sect:nonconst}

In this section we fix a pole $\p \in M$ and assume that the sectional curvatures $\mathcal{K}(x)$ are bounded below by a function $-c_m(\theta_x)$ which can grow unbounded as $\theta_x \to \infty$. We denote by $\psi_m$  the solution to the initial-value problem
\begin{equation}
\label{eqn:IVP-psim}
\begin{cases}
\psi_m''(\theta)=c_m(\theta)\psi_m(\theta), \qquad \forall \theta>0,\\[5pt]
\psi_m(0)=0,\quad \psi_m'(0)=1.
\end{cases}
\end{equation}
The main result, given by the theorem below, establishes sufficient conditions on the growth at infinity of the attractive potential (in terms of the growth of curvatures) that guarantee existence of a global energy minimizer.

\begin{theorem}[Existence of a global minimizer]
\label{thm:exist-gen} 
Let $M$ be an $\dm$-dimensional Cartan-Hadamard manifold and $\p \in M$ a fixed pole. Assume that the sectional curvatures of $M$ satisfy
\begin{equation}
\label{eqn:K-lb-nm}
-c_m(\theta_x) \leq \mathcal{K}(x;\sigma)\leq 0,
\end{equation}
for all $x \in M$ and all two-dimensional subspaces $\sigma \subset T_x M$, where $c_m(\cdot)$ is a positive, continuous and non-decreasing function. Also assume that $h:[0, \infty)\to[0, \infty)$ is a non-decreasing lower semi-continuous function which satisfies
\begin{equation}
\label{eqn:cond-h-exist}
h(\theta)\geq \phi(\theta),\qquad \text{ and } \qquad   \lim_{\theta\to\infty}\frac{\phi(\theta)}{\theta\sqrt{c_m(\theta)}}=\infty,
\end{equation}
for some convex function $\phi$.
Then there exists a global minimum of $E[\rho]$ in $\calP_\p(M)$. 
\end{theorem}
 
 \begin{remark}
 \label{rmk:weaker-cond}
 \normalfont
 In the proof of Theorem \ref{thm:exist-gen} we will use in fact a weaker assumption on the growth at infinity of $h$, namely that
 \begin{equation}
 \label{eqn:phi-relaxed}
h(\theta)\geq \phi(\theta),\qquad \text{ and } \qquad  \lim_{\theta\to\infty}\frac{\phi(\theta)}{\log\left(\frac{\psi_m(\theta)}{\theta}\right)}=\infty,
 \end{equation}
for some convex function $\phi$, where $\psi_m$ is the solution of the IVP \eqref{eqn:IVP-psim}. Nevertheless, we chose to present Theorem \ref{thm:exist-gen} with the explicit growth at infinity given in terms of the known function $c_m(\cdot)$, rather than using the unknown solution $\psi_m(\cdot)$ of the IVP.
 
To see that \eqref{eqn:phi-relaxed} is indeed weaker than \eqref{eqn:cond-h-exist}, one can argue as following. By Lemma \ref{estpsi}, it holds that
 \[
 \psi_m(\theta)\leq \psi_m(\theta_0)\exp\left((1+\epsilon)\int_{\theta_0}^\theta\sqrt{c_m(t)}\, \d t\right),\qquad \forall \theta > \theta_0,
 \]
for some $\epsilon>0$ and $\theta_0>0$. As $c_m$ is non-decreasing, we have
\[
\int_{\theta_0}^\theta\sqrt{c_m(t)}\, \d t\leq \theta\sqrt{c_m(\theta)}, \qquad \forall \theta > \theta_0.
\]
Then, we find 
\begin{equation}
\label{ineq:log}
\log\left(\frac{\psi_m(\theta)}{\theta}\right)\leq \log\left(\psi_m(\theta_0)\right)+(1+\epsilon)\theta\sqrt{c_m(\theta)}-\log\theta, \qquad \forall \theta>\theta_0.
\end{equation}

We also note that $\log \left(\frac{\psi_m(\theta)}{\theta} \right) \to \infty$ as $\theta \to \infty$; this follows from Lemma \ref{lem:linearlog}. Now, divide by $\log\left(\frac{\psi_m(\theta)}{\theta}\right)$ in \eqref{ineq:log}, and find that
\[
1\leq\frac{(1+\epsilon)\theta\sqrt{c_m(\theta)}}{\log\left(\frac{\psi_m(\theta)}{\theta}\right)}+\frac{\log(\psi_m(\theta_0))}{\log\left(\frac{\psi_m(\theta)}{\theta}\right)},
\]
for large $\theta$. Since the second term in the right-hand-side above vanishes in the limit $\theta \to \infty$, we have
\[
1\leq \lim_{\theta\to\infty}\frac{(1+\epsilon)\theta\sqrt{c_m(\theta)}}{\log\left(\frac{\psi_m(\theta)}{\theta}\right)}.
\]
Therefore, if $\phi$ satisfies the asymptotic condition in \eqref{eqn:cond-h-exist}, then it also satisfies the one in \eqref{eqn:phi-relaxed}.
 \end{remark}

We present the proof of Theorem \ref{thm:exist-gen} after we establish several important preparatory results.
\begin{lemma}[\cite{FePa2024}]
\label{Lem6.1}
Let $M$ be a Cartan-Hadamard manifold, $\p \in M$ a fixed pole and $\rho \in \calP_\p(M)$. Then we have
\[
(2-\sqrt{2})\calW_1(\rho, \delta_\p)\leq \iint_{M\times M} \dist(x, y)\rho(x)\rho(y)\d x\d y\leq 2\calW_1(\rho, \delta_\p).
\]
Here, $\delta_\p$ denotes the Dirac delta measure centred at the pole $\p$.
\end{lemma}
\begin{proof} This result was stated and proved in \cite[Lemma 6.1]{FePa2024}. The proof is based on Rauch comparison theorem and some simple calculations and estimates. We also note that this result does not require any information on the curvatures of $M$, and it applies to general Cartan-Hadamard manifolds.
\end{proof}

\begin{lemma}
\label{lemma:Hconvex}
Let $\psi_m$ be the solution to \eqref{eqn:IVP-psim}, where $c_m(\cdot)\geq0$ is a continuous non-decreasing function on $[0,\infty)$. Then
\begin{equation}
\label{eqn:H}
H(\theta):=\log\left(\frac{\psi_m(\theta)}{\theta}\right)
\end{equation}
is a non-decreasing convex function on $(0,\infty)$.
\end{lemma}

\begin{proof} Note first that by Remark \ref{rmk:psimM}, $\psi_m(\theta) \geq 0$ and $\psi_m'(\theta) \geq 1$ for all $\theta \geq 0$.
By direct calculation we get
\begin{equation}
\label{eqn:Hprime}
\frac{\d}{\d\theta}\log\left(\frac{\psi_m(\theta)}{\theta}\right)=\frac{\psi_m'(\theta)}{\psi_m(\theta)}-\frac{1}{\theta}.
\end{equation}
Set 
\[
f(\theta)=\frac{\psi_m(\theta)}{\psi_m'(\theta)}. 
\]
Then,
\begin{equation}
\label{eqn:fprime}
f'(\theta)=1-c_m(\theta)f(\theta)^2,
\end{equation}
and hence,
\[
f'(\theta)\leq1.
\]
This implies (also note that $f(0)=0$) that $f(\theta)\leq\theta$ for all $\theta>0$ and hence, by \eqref{eqn:Hprime}, we find $H'(\theta) \geq 0$. This shows $H$ is non-decreasing.

To show that  $H(\theta)$ is a convex function on $\theta>0$, we will show that its second derivative is positive. Compute
\begin{equation}
\label{eqn:Hdprime}
\begin{aligned}
\frac{\d^2}{\d\theta^2}\log\left(\frac{\psi_m(\theta)}{\theta}\right)
&=\frac{\psi_m''(\theta)\psi_m(\theta)-\psi_m'(\theta)^2}{\psi_m(\theta)^2}+\frac{1}{\theta^2}\\[2pt]
&=c_m(\theta)-\left(\frac{\psi_m'(\theta)}{\psi_m(\theta)}\right)^2+\frac{1}{\theta^2}.
\end{aligned}
\end{equation}
Fix $\theta_0>0$. Since $c_m$ is a continuous non-decreasing function, we have
\[
c_m(\theta_0)=\max_{0\leq \theta\leq \theta_0}c_m(\theta).
\]
Using the function $f$ again, we have by \eqref{eqn:fprime}:
\[
f'(\theta)\geq 1-c_m(\theta_0)f(\theta)^2, \qquad \text{ for } 0<\theta\leq \theta_0.
\]
Define $\tilde{f}(\theta)$ as the solution to
\[
\begin{cases}
\tilde{f}'(\theta)=1-c_m(\theta_0)\tilde{f}(\theta)^2,\qquad 0<\theta\leq \theta_0, \\[2pt]
\tilde{f}(0)=0.
\end{cases}
\]
By ODE comparison we then have
\begin{equation*}
f(\theta) \geq \tilde{f}(\theta)=\frac{\tanh(\sqrt{c_m(\theta_0)}\theta)}{\sqrt{c_m(\theta_0)}}, \qquad\forall~0<\theta\leq \theta_0.
\end{equation*}
Using this in \eqref{eqn:Hdprime} we then find
\begin{align*}
\frac{\d^2}{\d\theta^2}\log\left(\frac{\psi_m(\theta)}{\theta}\right)\bigg|_{\theta=\theta_0}&=c_m(\theta_0)-\left(\frac{\psi_m'(\theta_0)}{\psi_m(\theta_0)}\right)^2+\frac{1}{\theta^2_0}\\
&\geq c_m(\theta_0)-c_m(\theta_0)\coth^2(\sqrt{c_m(\theta_0)}\theta)+\frac{1}{\theta_0^2}\\
&=\frac{1}{\theta_0^2}-\frac{c_m(\theta_0)}{\sinh^2(\sqrt{c_m(\theta_0)}\theta_0)}\\
&>\frac{1}{\theta_0^2}-\frac{c_m(\theta_0)}{c_m(\theta_0)\theta_0^2}=0.
\end{align*}
where we used $\sinh x> x$ for all $x>0$ in the last inequality. This shows the convexity of the function $H$ on $(0,\infty)$.
\end{proof}

\begin{remark}
\label{rmk:Heven}
Let $\tilde{H}$ be the even extension of the function $H$ on $\bbr$:
\[
\tilde{H}(\theta)=\begin{cases}
H(\theta)\qquad&\text{ for }\theta>0, \\
0\qquad&\text{ for }\theta=0, \\
H(-\theta)\qquad&\text{ for }\theta<0.
\end{cases}
\]
It can be shown by direct calculations (not shown here) using L'H\^{o}pital theorem that $\tilde{H}$ is second order differentiable at $0$, with $\tilde{H}(0)=0$, $\tilde{H}'(0)=0$, and $\tilde{H}''(0)=\frac{c_m(0)}{3}>0$. Consequently, $\tilde{H}$ is a globally convex function on $\bbr$.
\end{remark}

\begin{proposition}
\label{prop:E-W1-gen}
Let $M$ be an $\dm$-dimensional Cartan-Hadamard manifold whose sectional curvatures satisfy \eqref{eqn:K-lb-nm}, with $\p \in M$ a fixed pole, and $c_m(\cdot)$ a positive, continuous and non-decreasing function. Also assume that $h:[0,\infty) \to [0,\infty)$ is lower semi-continuous and satisfies \eqref{eqn:phi-relaxed}. Then there exist constants $C_0$ and $\calC$ such that
\begin{equation*}
E[\rho] \geq-C_0 \dm+\frac{\mathcal{C}}{2}+\frac{2-\sqrt{2}}{2}\calW_1(\rho, \delta_\p), \qquad \text{ for all } \rho \in \calP_\p(M).
\end{equation*}
Specifically, $C_0$ depends on $\dm$, while $\calC$ depends on $h$, $c_m$ and $\dm$.
\end{proposition}

\begin{proof}
By Rauch comparison theorem (see \eqref{eqn:dist-ineq}) and the law of cosines, for any two points $x,y \in M$, we have
\begin{align*}
\dist(x, y)^2 &\geq |\log_\p x- \log_\p y |^2 \\
& =  \theta_x^2+\theta_y^2-2\theta_x\theta_y\cos\angle(x\p y) \\
& \geq(\theta_x-\theta_y\cos\angle(x\p y))^2.
\end{align*}
This yields
\begin{equation}
\label{eqn:dxy-ineq2}
\dist(x, y)\geq  |\theta_x-\theta_y\cos\angle(x\p y)|.
\end{equation}

Take any $\rho \in \calP_\p(M)$. Consider the function $H$ defined in \eqref{eqn:H} and its even extension (see Remark \ref{rmk:Heven}). Since $H$ is non-decreasing (by Lemma \ref{lemma:Hconvex}), along with \eqref{eqn:dxy-ineq2}, one can estimate
\begin{equation}
\label{eqn:iint-ineq1}
\begin{aligned}
\iint_{M \times M} H(\dist(x, y))\rho(x)\rho(y)\d x \d y &
\geq\iint_{M \times M}  H(|\theta_x-\theta_y\cos\angle(x\p y)|)\rho(x)\rho(y)\d x \d y\\
&=\iint_{M \times M}  \tilde{H}(\theta_x-\theta_y\cos\angle(x\p y))\rho(x)\rho(y)\d x \d y.
\end{aligned}
\end{equation}

Since $\rho \in \calP_\p(M)$, we have
\[
\int_M \log_\p y\rho(y)\d y=0.
\]
In this equation, take the inner product with $\log_\p x$ for some arbitrary $x\neq \p$, to get
\[
\int_M \theta_x\theta_y\cos\angle(x\p y) \rho(y)\d y=0,
\]
which is equivalent to
\begin{equation}
\label{eqn:int-thetay}
\int_M \theta_y\cos\angle(x\p y)\rho(y) \d y=0.
\end{equation}

Now, use the convexity of $\tilde{H}$ and Jensen's inequality, together with \eqref{eqn:int-thetay}, to obtain
\begin{align*}
\int_M \tilde{H}(\theta_x-\theta_y\cos\angle(x\p y))\rho(y)\d y & \geq \tilde{H}\left(\int_M (\theta_x-\theta_y\cos\angle(x\p y))\rho(y)\d y \right) \\[2pt]
& = \tilde{H}(\theta_x),
\end{align*}
for any $x \neq \p$. By integrating the inequality above with respect to $\rho(x)$ we then find
\begin{equation}
\label{eqn:iint-ineq2}
\iint_{M \times M} \tilde{H}(\theta_x-\theta_y\cos\angle(x\p y))\rho(x)\rho(y)\d x \d y\geq \int_M \tilde{H}(\theta_x)\rho(x)\d x.
\end{equation}

Finally, combine \eqref{eqn:iint-ineq1} with \eqref{eqn:iint-ineq2} and use the expression of $H$ in \eqref{eqn:H}, to get
\begin{equation}
\label{ineq:dint-int}
\iint_{M \times M} \log\left(\frac{\psi_m(\dist(x, y))}{\dist(x, y)}\right)\rho(x)\rho(y)\d x \d y\geq\int_M \log\left(\frac{\psi_m(\theta_x)}{\theta_x}\right)\rho(x)\d x.
\end{equation}
Using \eqref{ineq:dint-int} in Theorem \ref{HLSub} (see \eqref{eqn:HLS-CH-II}) we then reach
\begin{equation}
\label{eqn:ineq-key-gen}
\begin{aligned}
&-\iint_{M\times M}\log(\dist(x, y))\rho(x)\rho(y)\d x \d y\\
&\hspace{2cm}  \leq  \frac{1}{\dm}\int_M \rho(x)\log\rho(x)\d x+\frac{(\dm-1)}{\dm}\iint_{M \times M} \log\left(\frac{\psi_m(\dist(x, y))}{\dist(x, y)}\right)\rho(x)\rho(y)\d x \d y+C_0.
\end{aligned}
\end{equation}

Since $h$ satisfies \eqref{eqn:phi-relaxed}, we have
\[
\lim_{\theta\to\infty}\left(\frac{h(\theta)}{2}-2(\dm-1)\log\left(\frac{\psi_m(\theta)}{\theta}\right)\right)=\infty.
\]
On the other hand, by Lemma \ref{lem:linearlog} we infer that $\log\left(\frac{\psi_m(\theta)}{\theta}\right) $ grows at least linearly at infinity. Then, from \eqref{eqn:phi-relaxed} we conclude that $\phi$ (and hence $h$) have {\em superlinear} growth at infinity, which yields
\[
\lim_{\theta\to\infty}\left(\frac{h(\theta)}{2}-2\dm\log\theta-\theta\right)=\infty.
\]
By combining the two limits above we get
\[
\lim_{\theta\to\infty}\left(h(\theta)-2(\dm-1)\log\left(\frac{\psi_m(\theta)}{\theta}\right)-2\dm\log\theta-\theta\right)=\infty.
\]
Now we note the behaviour at zero of this expression, given by
\[
\lim_{\theta\to0^+}\left(h(\theta)-2(\dm-1)\log\left(\frac{\psi_m(\theta)}{\theta}\right)-2\dm\log\theta-\theta\right)=\infty,
\]
where we used that $h$ is bounded from below and $\psi_m(\theta) \sim \theta$ near $\theta =0$. By the lower semi-continuity of $h$, we can then conclude that there exists a constant $\mathcal{C}$ such that
\begin{equation}
\label{eqn:h-lb1}
h(\theta) - 2(\dm-1)\log\left(\frac{\psi_m(\theta)}{\theta}\right)- 2\dm\log\theta- \theta \geq \mathcal{C}, \qquad \text{ for all } \theta>0.
\end{equation}

Finally, from \eqref{eqn:h-lb1} and \eqref{eqn:ineq-key-gen}, we get 
\begin{align*}
E[\rho]&= \int_M\rho(x)\log\rho(x)\d x + \frac{1}{2} \iint_{M \times M}  h (\dist(x, y))\rho(x)\rho(y)\d x\d y \\
& \geq \int_M\rho(x)\log\rho(x)\d x+(\dm-1)\iint_{M \times M}\log\left(\frac{\psi_m(\dist(x,y))}{\dist(x, y)}\right)\rho(x)\rho(y)\d x\d y \\
& \quad +\dm\iint_{M \times M} \log(\dist(x, y))\rho(x)\rho(y)\d x\d y +\frac{1}{2}\iint_{M \times M} \dist(x, y)\rho(x)\rho(y)\d x\d y+\frac{\mathcal{C}}{2}\\
&\geq -C_0 \dm +\frac{\mathcal{C}}{2}+\frac{1}{2}\iint_{M \times M}\dist(x, y)\rho(x)\rho(y)\d x\d y.
\end{align*}
The conclusion now follows from the estimate above and Lemma \ref{Lem6.1}.
\end{proof}

\begin{lemma}\label{L-Z-1}
(i) Assume $M$ is a Cartan-Hadamard manifold and $h:[0,\infty) \to [0,\infty)$ is lower semi-continuous. Then, the interaction energy is lower semi-continuous with respect to weak convergence, i.e., if $\rho_k \rightharpoonup \rho_0$ weakly (as measures) as $k \to \infty$, it holds that
\[
\iint_{M\times M}h(\dist(x, y)) \rho_0(x)\rho_0(y)\d x\d y\leq \liminf_{k \to \infty} \iint_{M\times M}h(\dist(x, y)) \rho_k(x)\rho_k(y)\d x\d y.
\]
(ii) Assume in addition that $M$ and $h$ satisfy the assumptions in Proposition \ref{prop:E-W1-gen}, with $\p$ a fixed pole in $M$, and suppose $\{\rho_k\}_{k\ge 1} \subset \calP_{\p}(M)$ is a sequence of probability densities such that $E[\rho_k]$ is uniformly bounded from above. Then, 
\[
\iint_{M \times M} h(\dist(x, y))\rho_k(x)\rho_k(y)\d x\d y
\]
is also uniformly bounded from above.
\end{lemma}
\begin{proof}
{\em Part (i).}  Using the exponential map $f^{-1} = \exp_\p $ to make a change of variable, we can write the interaction energy of a density $\rho$ as 
\begin{align*}
& \iint_{M\times M}h(\dist(x, y))\rho(x)\rho(y)\d x\d y \\
& \qquad =\iint_{T_\p M\times T_\p M}h(\dist(\exp_\p u,\exp_\p v ))\log|J(f^{-1})(u)|\log|J(f^{-1})(v)| f_\#\rho(u) f_\#\rho(v) \d u\d v.
\end{align*}
The right-hand-side above can be interpreted as the interaction energy of the density $f_\# \rho$ on $T_\p M\simeq \mathbb{R}^\dm$ with the interaction potential $\tilde{h} : T_\p M \times T_\p M \to \mathbb{R}^\dm$ given by
\[
\tilde{h}(u, v)=h(\dist(\exp_\p u,\exp_\p v ))\log|J(f^{-1})(u)|\log|J(f^{-1})(v)|.
\]
Note that $\tilde{h}$ is a lower semi-continuous function, as $h$ is lower semi-continuous and $\exp_\p$ is a diffeomorphism.

Consider now a sequence $\{\rho_k\}_{k\ge 1}$, with $\rho_k \rightharpoonup \rho_0$ weakly (as measures) as $k \to \infty$. By the change of variable formula, it is straightforward to show that $\rho_{k} \rightharpoonup \rho_0$ implies $f_\# \rho_{k} \rightharpoonup f_\# \rho_0$ as $k \to \infty$.  In the Euclidean space, the interaction energy is lower semi-continuous with respect to weak convergence provided the interaction potential is lower semi-continuous \cite[Proposition 7.2]{Santambrogio2015}. By combining all these facts, we then find
\begin{equation*}
\begin{aligned}
\liminf_{k \to \infty }\iint_{M\times M}h(\dist(x, y))\rho_k(x)\rho_k(y)\d x\d y  &= \liminf_{k \to \infty} \iint_{T_\p M\times T_\p M}\tilde{h}(u,v) f_\#\rho_k(u) f_\#\rho_k(v) \d u\d v \\[2pt]
& \geq \iint_{T_\p M\times T_\p M}\tilde{h}(u,v) f_\#\rho_0(u) f_\#\rho_0(v) \d u\d v \\[2pt]
&= \iint_{M\times M}h(\dist(x, y))\rho_0(x)\rho_0(y)\d x\d y.
\end{aligned}
\end{equation*}

{\em Part (ii).}  Since by \eqref{eqn:phi-relaxed},
\[
\lim_{\theta\to\infty}\frac{h(\theta)}{\log\left(\frac{\psi_m(\theta)}{\theta}\right)}\geq \lim_{\theta\to\infty}\frac{\phi(\theta)}{\log\left(\frac{\psi_m(\theta)}{\theta}\right)}=\infty,
\]
we infer that $h(\theta)$ satisfies
\[
\lim_{\theta\to\infty}\left(\frac{h(\theta)}{2}-4(\dm-1)\log\left(\frac{\psi_m(\theta)}{\theta}\right)\right) = \infty.
\]
Also, by $h\geq\phi$ and the fact that $\phi$ is superlinear (see Lemma \ref{lem:linearlog}), we have
\[
\lim_{\theta\to\infty}\left(\frac{h(\theta)}{2}-4\dm\log\theta\right) = \infty.
\]
Now combine the two limits to get
\[
\lim_{\theta\to\infty}\left(h(\theta)-4(\dm-1)\log\left(\frac{\psi_m(\theta)}{\theta}\right)-4\dm\log\theta\right) = \infty.
\]

On the other hand, the dominant term of this expression as $\theta \to 0^+$ is $-4 \dm \log \theta$, and hence
\[
\lim_{\theta\to0^+}\left(h(\theta)-4(\dm-1)\log\left(\frac{\psi_m(\theta)}{\theta}\right)-4\dm\log\theta\right) = \infty.
\]
Since $h$ is lower semi-continuous, there exists then a constant $\mathcal{C}'$ such that
\[
h(\theta) - 4(\dm-1)\log\left(\frac{\psi_m(\theta)}{\theta}\right)- 4\dm\log\theta \geq 2\mathcal{C}', \qquad \text{ for all } \theta>0,
\]
which implies
\begin{equation}
\label{eqn:h-lb2}
h(\theta)\geq\frac{h(\theta)}{2}+2(\dm-1)\log\left(\frac{\psi_m(\theta)}{\theta}\right)+2\dm\log\theta+\mathcal{C}',  \qquad \text{ for all } \theta>0.
\end{equation}

Finally, using \eqref{eqn:h-lb2} and the consequence of the HLS inequality given by \eqref{eqn:ineq-key-gen}, we estimate (for all $k \geq 1$):
\begin{align*}
E[\rho_k] &= \int_M\rho_k(x)\log\rho_k(x)\d x + \frac{1}{2} \iint_{M \times M}  h (\dist(x, y))\rho_k(x)\rho_k(y)\d x\d y \\
&\geq\int_M\rho_k(x)\log\rho_k(x)\d x+(\dm-1)\iint_{M \times M} \log\left(\frac{\psi_m(\dist(x,y))}{\dist(x, y)}\right)\rho_k(x)\rho_k(y)\d x\d y\\
&\qquad+\dm\iint_{M \times M} \log(\dist(x, y))\rho_k(x)\rho_k(y)\d x\d y+\frac{\mathcal{C}'}{2}+\frac{1}{4}\iint_{M \times M} h(\dist(x, y))\rho_k(x)\rho_k(y)\d x\d y\\
&\geq -C_0\dm + \frac{\mathcal{C}'}{2} + \frac{1}{4}\iint_{M \times M} h(\dist(x, y))\rho_k(x)\rho_k(y)\d x\d y.
\end{align*}
Since $E[\rho_k]$ is uniformly bounded from above, we infer the conclusion.
\end{proof}

\begin{lemma}\label{L-Z-2}
Assume $M$ and $h$ satisfy the assumptions in Proposition \ref{prop:E-W1-gen}, with $\p$ a fixed pole in $M$. Suppose $\{\rho_k\}_{k \geq 1} \subset \calP_{\p}(M)$ is a sequence of probability densities such that $E[\rho_k]$ is uniformly bounded from above, and $\rho_k \rightharpoonup \rho_0$ weakly (as measures) as $k \to \infty$. Then, we have
\[
\lim_{k \to \infty} \int_M \rho_k(x)\log |J(f^{-1})(f(x))|\d x =\int_M \rho_0(x)\log |J(f^{-1})(f(x))|\d x,
\]
where $f$ denotes the Riemannian logarithm map at $\p$ (see \eqref{eqn:Rlog}).
\end{lemma}
\begin{proof}
Since $E[\rho_k]$ is uniformly bounded from above, by Lemma \ref{L-Z-1} part (ii), so is \\ $\iint_{M \times M} h(\dist(x, y))\rho_k(x)\rho_k(y)\d x\d y$. Denote by $U$ such an upper bound, i.e., 
\[
\iint_{M \times M} h(\dist(x, y))\rho_k(x)\rho_k(y)\d x\d y \leq U, \qquad \text{ for all } k \geq 1.
\]
Also, by Lemma  \ref{L-Z-1} part (i), we have
\[
\iint_{M \times M} h(\dist(x, y))\rho_0(x)\rho_0(y)\d x\d y \leq U.
\]

Now recall that $\phi$ in \eqref{eqn:cond-h-exist} is convex. By the same convex function arguments used in the proof of Proposition \ref{prop:E-W1-gen} (see the arguments used for function $H$ that led to \eqref{ineq:dint-int}), one can get
\begin{equation*}
\iint_{M \times M} \phi(\dist(x, y))\rho_k(x)\rho_k(y)\d x\d y\geq \int_M \phi(\theta_x)\rho_k(x)\d x, \qquad \text{ for all } k \geq 1,
\end{equation*}
and
\begin{equation*}
\iint_{M \times M} \phi(\dist(x, y))\rho_0(x)\rho_0(y)\d x\d y\geq \int_M \phi(\theta_x)\rho_0(x)\d x.
\end{equation*}
Consequently, since $h(\theta)\geq \phi(\theta)$, we infer from above that
\begin{equation}\label{phibounded}
\int_M \phi(\theta_x)\rho_k(x)\d x \leq U, \quad \text{ for all } k \geq 1, \qquad \text{ and} \quad \int_M \phi(\theta_x)\rho_0(x)\d x \leq U.
\end{equation}

For any $R>0$ fixed, we have the following estimate by the triangle inequality:
\begin{equation}
\label{eqn:int-diff}
\begin{aligned}
&\left|\int_M \rho_0(x)\log|J(f^{-1})(f(x))|\d x-\int_M \rho_k(x)\log|J(f^{-1})(f(x))|\d x\right|\\
& \qquad \leq \left|\int_{\theta_x\leq R}\rho_0(x)\log|J(f^{-1})(f(x))|\d x-\int_{\theta_x\leq R}\rho_k(x)\log|J(f^{-1})(f(x))|\d x\right|\\
& \qquad \quad +\int_{\theta_x>R}\rho_0(x)\log|J(f^{-1})(f(x))|\d x+\int_{\theta_x>R}\rho_k(x)\log|J(f^{-1})(f(x))|\d x,
\end{aligned}
\end{equation}
where we also used that $|J(f^{-1})(f(x))| \geq 1$ (a consequence of Theorem \ref{lemma:Chavel-thms-gen}). Using Theorem \ref{lemma:Chavel-thms-gen} and \eqref{phibounded}, we can estimate the last two terms in the right-hand-side above as follows:
\begin{align*}
&\int_{\theta_x>R}\rho_0(x)\log|J(f^{-1})(f(x))|\d x+\int_{\theta_x>R}\rho_k(x)\log|J(f^{-1})(f(x))|\d x\\[2pt]
& \quad \leq(\dm-1)\int_{\theta_x>R}\rho_0(x)\log\left(\frac{\psi_m(\theta_x)}{\theta_x}\right)\d x+(\dm-1)\int_{\theta_x>R}\rho_k(x)\log\left(\frac{\psi_m(\theta_x)}{\theta_x}\right)\d x\\[2pt]
& \quad \leq (\dm-1)\left\|\frac{\log\left(\frac{\psi_m(\cdot)}{\cdot}\right)}{\phi(\cdot)}\right\|_{L^\infty([R, \infty))}\left(
\int_{\theta_x>R}\rho_0(x)\phi(\theta_x)\d x+\int_{\theta_x>R}\rho_k(x)\phi(\theta_x)\d x\right)\\
& \quad \leq 2(\dm-1)\left\|\frac{\log\left(\frac{\psi_m(\cdot)}{\cdot}\right)}{\phi(\cdot)}\right\|_{L^\infty([R, \infty))}U.
\end{align*}

Using this result in \eqref{eqn:int-diff}, together with the weak convergence $\rho_k \rightharpoonup \rho_0$,  we then find
\begin{align*}
&\limsup_{k\to\infty}\left|\int_M \rho_0(x)\log|J(f^{-1})(f(x))|\d x-\int_M \rho_k(x)\log|J(f^{-1})(f(x))|\d x\right|\\[2pt]
&\qquad \leq \limsup_{k\to\infty}\left|\int_{\theta_x\leq R}\rho_0(x)\log|J(f^{-1})(f(x))|\d x-\int_{\theta_x\leq R}\rho_k(x)\log|J(f^{-1})(f(x))|\d x\right|\\[2pt]
&\qquad \quad +2(\dm-1)\left\|\frac{\log\left(\frac{\psi_m(\cdot)}{\cdot}\right)}{\phi(\cdot)}\right\|_{L^\infty([R, \infty))}U\\[2pt]
&\qquad = 2(\dm-1)\left\|\frac{\log\left(\frac{\psi_m(\cdot)}{\cdot}\right)}{\phi(\cdot)}\right\|_{L^\infty([R, \infty))}U.
\end{align*}

Finally, by assumption \eqref{eqn:phi-relaxed}, we have
\[
\lim_{R\to\infty}\left\|\frac{\log\left(\frac{\psi_m(\cdot)}{\cdot}\right)}{\phi(\cdot)}\right\|_{L^\infty([R, \infty))}=0.
\]
Since $R$ is arbitrary, we then get
\[
\limsup_{k\to\infty}\left|\int_M \rho_0(x)\log|J(f^{-1})(f(x))|\d x-\int_M \rho_k(x)\log|J(f^{-1})(f(x))|\d x\right|\leq0,
\]
which yields the desired conclusion.
\end{proof}

We can now show that the energy functional is lower semi-continuous along minimizing sequences.

\begin{proposition}[Lower semi-continuity of the energy]\label{lscE}
\label{prop:lsc-energy}
Let $\p \in M$ be a fixed pole and assume $M$ and $h$ satisfy the assumptions in Proposition \ref{prop:E-W1-gen}. Suppose $\{\rho_k\}_{k \geq 1} \subset \calP_{\p}(M)$ such that $E[\rho_k]$ is uniformly bounded from above, and $\rho_k \rightharpoonup \rho_0$ weakly (as measures) as $k \to \infty$. Then, the energy functional is lower semi-continuous along $\rho_k$, i.e.,
\[
E[\rho_0] \leq \liminf_{k\to\infty}E[\rho_k].
\]
\end{proposition}
\begin{proof}
The energy functional contains tho components: the entropy and the interaction energy; we will consider the two parts separately. The lower semi-continuity of the interaction energy was established in Lemma \ref{L-Z-1} part (i). For the entropy component, we use again the exponential map $f^{-1} = \exp_\p$ to change variables and write (see also \eqref{eqn:push-rho}):
\[
\int_M \rho_k(x)\log\rho_k(x)\d x=\int_{\bbr^\dm}f_\#\rho_k(u)\log f_\#\rho_k(u) \d u-\int_M \rho_k(x)\log |J(f^{-1})(f(x))|\d x,
\]
for all $k \geq 1$.

The first term in the right-hand-side above is the entropy of $f_\# \rho_k$ on $T_\p M\simeq \mathbb{R}^\dm$.  Also note that by the change of variable formula, $\rho_{k} \rightharpoonup \rho_0$ implies $f_\# \rho_{k} \rightharpoonup f_\# \rho_0$ as $k \to \infty$. 
Then, use the fact that the entropy functional in $\bbr^\dm$ is lower semi-continuous \cite{CaDePa2019}, along with Lemma \ref{L-Z-2}, to get
\begin{equation*}
\begin{aligned}
&  \liminf_{k \to\infty} \int_M \rho_{k}(x)\log\rho_{k}(x)\d x \\
& \qquad \geq  \liminf_{k \to \infty} \int_{\bbr^\dm}f_\#\rho_k(u)\log f_\#\rho_k(u) \d u- \lim_{k \to \infty} \int_M \rho_k(x)\log |J(f^{-1})(f(x))|\d x \\
& \qquad \geq  \int_{\bbr^\dm}f_\#\rho_0(u)\log f_\#\rho_0(u) \d u-\int_M \rho_0(x)\log |J(f^{-1})(f(x))|\d x\\
& \qquad = \int_M \rho_0(x)\log\rho_0(x)\d x,
 \end{aligned}
\end{equation*}
where for the equal sign we used again the change of variable $x= \exp_\p u$ and \eqref{eqn:push-rho}.

The lower semi-continuity of the energy now follows from the lower semi-continuity of its two components, as
\begin{align*}
\liminf_{k\to\infty}E[\rho_k] &\geq  \liminf_{k \to\infty} \int_M \rho_{k}(x)\log\rho_{k}(x)\d x +  \liminf_{k \to \infty} \frac{1}{2} \iint_{M\times M}h(\dist(x, y)) \rho_k(x)\rho_k(y)\d x\d y \\
& \geq  \int_M \rho_0(x)\log\rho_0(x)\d x +  \frac{1}{2} \iint_{M\times M}h(\dist(x, y)) \rho_0(x)\rho_0(y)\d x\d y \\
& = E[\rho_0].
\end{align*}
\end{proof}

The last ingredient needed for the proof of the main theorem is the conservation of centre of mass, given by the following result established in \cite{FePa2024}. We point out that in the next lemma, as well as in the proof of Theorem \ref{thm:exist-gen} that follows, we will use the notation $B_R(\delta_\p)$ for the geodesic ball in the space $(\calP_1(M),\calW_1)$, of radius $R$ and centre at $\delta_\p$. A similar notation has been used in the paper for the open ball $B_\theta(\p)$ of radius $\theta$ and centre at $\p$, in the geodesic space $(M,\dist)$. Nevertheless, the different spaces in which these geodesic balls are considered, will be clear from the context. 

\begin{lemma}[Conservation of centre of mass \cite{FePa2024}]
\label{lemma:CM}
Let $M$ be a Cartan-Hadamard manifold, $\p \in M$ a fixed pole, $R>0$ a fixed radius,  and $\{\rho_k\}_{k\geq 1}$ a sequence in $\overline{B_R(\delta_\p)}\cap \mathcal{P}_\p(M)$. Also assume that
\[
\int_M \phi(\theta_x) \rho_k(x) \d x
\]
is uniformly bounded from above, where $\phi$ is a function with superlinear growth at infinity, i.e., $\lim_{\theta \to \infty} \frac{\phi(\theta)}{\theta} = \infty$. Then, there exists a subsequence of $\{\rho_k\}_{k\geq 1}$ which converges weakly as measures to $\rho_0\in \overline{B_R(\delta_\p)}\cap \mathcal{P}_\p(M)$.
\end{lemma}
\begin{proof}
The result was stated and proved in \cite[Lemma 6.7]{FePa2024}. First, by tightness and Prokhorov's theorem, one can extract a subsequence of $\{\rho_k\}_{k\geq 1}$ which converges weakly as measures to $\rho_0\in \overline{B_R(\delta_\p)}$. Then, use the uniform boundedness of $\int_M \phi(\theta_x) \rho_k(x) \d x$ to argue that the centre of mass is preserved in the limit and hence, $\rho_0 \in \calP_\p(M)$.  We refer to \cite[Lemma 6.7]{FePa2024} for the details.
\end{proof}
We now present the proof of Theorem \ref{thm:exist-gen}, one of the main results of this paper.

\smallskip
{\em Proof of Theorem \ref{thm:exist-gen}.} 

We will present the proof in several steps.

{\em \underline{Step 1.}} The energy is bounded below on $\calP_\p(M)$ (see Proposition \ref{prop:E-W1-gen}). In particular, $E[\rho]$ is bounded below on $\overline{B_1(\delta_\p)}\cap \mathcal{P}_\p(M)$, and define
\[
E_0:=\inf_{\rho\in \overline{B_1(\delta_\p)} \cap \mathcal{P}_\p(M)} E[\rho].
\]
Take a minimizing sequence $\{\rho_k\}_{k\in\mathbb{N}}$ of $E[\rho]$ on $\overline{B_1(\delta_\p)}\cap \mathcal{P}_\p(M)$, i.e.,
\[
\{\rho_k\}\subset \overline{B_1(\delta_\p)}\cap \mathcal{P}_\p(M), \qquad \lim_{k\to\infty}E[\rho_k]=E_0.
\]
Without loss of generality, we can assume that $E[\rho_k]$ is decreasing. In particular, $E[\rho_k]$ is uniformly bounded from above and by the argument used in the proof of Lemma \ref{L-Z-2} (see the derivation of \eqref{phibounded}), $\int_M \phi(\theta_x) \rho_k(x) \d x $ is also uniformly bounded from above.  We also recall that the function $\phi$ used in our assumption \eqref{eqn:phi-relaxed} (or \eqref{eqn:cond-h-exist}) for $h$ has superlinear growth at infinity, and therefore it can be used in the context of Lemma \ref{lemma:CM}. Hence, by Lemma \ref{lemma:CM}, there exists a subsequence $\{\rho_{k_l}\}_{l\in\mathbb{N}}$ which converges weakly to $\rho_0 \in \overline{B_1(\delta_\p)}\cap \mathcal{P}_\p(M)$. 

From the lower semi-continuity of the energy in Proposition \ref{prop:lsc-energy}, we have
\begin{equation*}
\lim_{l\to\infty}E[\rho_{k_l}]\geq E[\rho_0].
\end{equation*}
Therefore,
\[
E_0=\lim_{l\to\infty}E[\rho_{k_l}]\geq E[\rho_0]\geq E_0,
\]
and hence $E[\rho_0]= E_0$, i.e, $\rho_0$ is a minimizer of the energy on $\overline{B_1(\delta_\p)}\cap \mathcal{P}_\p(M)$.

{\em \underline{Step 2.}} Take $R>1$ large enough such that
\[
E[\rho_0]\leq -C_0\dm+\frac{\calC}{2}+ \frac{2-\sqrt{2}}{2} R,
\]
where $C_0$ and $\calC$ are the constants from Proposition \ref{prop:E-W1-gen} ($C_0$ is the constant from the HLS inequality \eqref{eqn:HLS-CH-II}, and depends only on the dimension $\dm$, and $\calC$ depends on $h$, $c_m$ and $\dm$). Then, by Proposition \ref{prop:E-W1-gen}, for any $\rho\in  \mathcal{P}_\p(M)\setminus \overline{B_{R}(\delta_\p)}$ it holds that
\begin{equation}
\label{eqn:E-est-BRc}
\begin{aligned}
E[\rho] &\geq -C_0\dm+\frac{\calC}{2}+\frac{2-\sqrt{2}}{2}\calW_1(\rho, \delta_\p) \\[2pt]
&>  -C_0 \dm+\frac{\calC}{2}+\frac{2-\sqrt{2}}{2}R \\[2pt]
& \geq E[\rho_0].
\end{aligned}
\end{equation}

On the other hand,
\[
E[\rho_0] = \inf_{\rho\in \overline{B_{1}(\delta_\p)}\cap \mathcal{P}_\p(M)}E[\rho] \geq \inf_{\rho\in \overline{B_{R}(\delta_\p)}\cap \mathcal{P}_\p(M)}E[\rho],
\]
which together with \eqref{eqn:E-est-BRc}, it implies
\[
\inf_{\rho\in \mathcal{P}_\p(M)}E[\rho] = \inf_{\rho\in \overline{B_{R}(\delta_\p)}\cap \mathcal{P}_\p(M)}E[\rho].
\]
Finally, the existence of a global minimizer of $E[\rho]$ on $\overline{B_{R}(\delta_\p)}\cap \mathcal{P}_\p(M)$ can be argued exactly as we argued above the existence of the minimizer $\rho_0$ in $\overline{B_{1}(\delta_\p)}\cap \mathcal{P}_\p(M)$. This concludes the proof.

\hspace {14cm} $\qed$

\begin{example}
\label{ex:exist}
\normalfont
We will revisit here the examples from Example \ref{ex:nonexist} in the context of Theorem \ref{thm:exist-gen}. 
\begin{enumerate}
\item Constant: $c_m(\theta) \equiv c_m$.  In this case, \eqref{eqn:cond-h-exist} reduces to $h \geq \phi$, with $\phi$ growing superlinearly at infinity, which recovers the result in \cite[Theorem 2.2]{FePa2024}.

\item Power law: $c_m(\theta)=\theta^k$ with $k\geq1$. Ground states exists provided $h \geq \phi$ and $\phi$ grows faster than $\theta^{k/2+1}$ at infinity. Note that the result is sharp in the sense that a growth of $h$ slower than $\theta^{k/2+1}$ is not sufficient to contain the diffusion, and leads to spreading (see item (2) in Example \ref{ex:nonexist}).

\item Exponential growth: $c_m(\theta)= e^{\beta\theta}$ with $\beta >0$. For this case, spreading is prevented provided $h(\theta)$ grows at infinity faster than $e^{\beta\theta/2}\theta$. Hence, an exponential growth of the attractive interactions is needed to contain diffusion on a manifold with exponentially growing curvatures.
\end{enumerate}
\end{example}

We conclude the paper with a discussion on the assumption of the monotonicity of $c_m$. This assumption can be dropped by requiring a stronger assumption on the behaviour at infinity of $h$. Indeed, assume that the sectional curvatures satisfy \eqref{eqn:K-lb-nm}, where $c_m$ is not necessarily non-decreasing. Define
\[
\tilde{c}_m(\theta):=\max_{0\leq t \leq \theta}c_m(t).
\]
Then, $\tilde{c}_m$ is continuous and non-decreasing and by \eqref{eqn:K-lb-nm}, the curvatures satisfy
\[
-\tilde{c}_m(\theta_x) \leq \mathcal{K}(x;\sigma)\leq 0,
\]
for all $x \in M$ and all two-dimensional subspaces $\sigma \subset T_x M$. Hence, we can apply Theorem \ref{thm:exist-gen} with $\tilde{c}_m$ instead of $c_m$, and obtain the following corollary.

\begin{corollary}
Let $M$ be an $\dm$-dimensional Cartan-Hadamard manifold and $\p \in M$ a fixed pole. Assume that the sectional curvatures of $M$ satisfy \eqref{eqn:K-lb-nm}, where $c_m(\cdot)$ is a positive, continuous function. Also assume that $h:[0, \infty)\to[0, \infty)$ is a non-decreasing lower semi-continuous function which satisfies
\begin{equation}
h(\theta)\geq \phi(\theta),\qquad \text{ and } \qquad   \lim_{\theta\to\infty}\frac{\phi(\theta)}{\theta\sqrt{\sup_{0\leq t\leq \theta}c_m(t)}}=\infty,
\end{equation}
for some convex function $\phi$.
Then there exists a global minimum of $E[\rho]$ in $\calP_\p(M)$. 
\end{corollary}


\appendix
\section{Proof of Theorem \ref{lemma:Chavel-thms-gen}}
\label{appendix:Chavel-gen}

\noindent {\em Part I: Left inequality.} This part follows exactly the steps of the proof of \cite[Theorem III.4.1]{Chavel2006} (constant bounds of $\mathcal{K}$), but we will reproduce it here for completeness. Following the proof of \cite[Theorem II.6.4]{Chavel2006}, for $Y\in \mathcal{J}^\perp$, one has
\[
|Y|' = \langle Y, \nabla_t Y \rangle |Y|^{-1},
\]
and then,
\begin{align}
|Y|'' &= |Y|^{-1} \left( |\nabla_t Y|^2 - \langle Y, \calR(\gamma',Y)\gamma'\rangle \right) - |Y|^{-3} \langle Y, \nabla_t Y \rangle^2 \nonumber \\[2pt]
& \geq c_M |Y| + |Y|^{-3} \left( |\nabla_t Y|^2 |Y|^2 - \langle Y, \nabla_t Y \rangle^2 \right) \nonumber \\[2pt]
& \geq c_M |Y|,
\label{eqn:Ydp-ineq}
\end{align}
where in the first line we used the Jacobi equation \eqref{eqn:Jacobi}, for the second line we used the upper bound on the sectional curvatures, and the third line follows from Cauchy-Schwarz inequality.

We set
\[
|Y|(0)=0,\quad |Y|'(0)=1.
\]
It then follows from \eqref{eqn:Ydp-ineq} and \eqref{eqn:psimM} that $F:= |Y|' \psi_M - |Y| \psi_M'$ is non-decreasing, and hence, $F \geq 0$ (as $F(0) = 0$). We can write this as 
\begin{equation}
\frac{ |Y|' }{ |Y|} \geq \frac{\psi_M'}{\psi_M}.
\label{eqn:YpoY}
\end{equation}


Instead of calculating the determinant of $\mathcal{A}$, we work with the self-adjoint matrix $\mathcal{B}$ defined as
\[
\mathcal{B}:=\mathcal{A}^*\mathcal{A}.
\]
From this definition, we have
\begin{equation}
\frac{(\det\mathcal{A})'}{\det\mathcal{A}}=\frac{1}{2}\frac{(\det\mathcal{B})'}{\det\mathcal{B}} = \frac{1}{2} \mathrm{tr} (\mathcal{B}'\mathcal{B}^{-1}),
\label{eqn:detApodetA}
\end{equation}
where the second equal sign comes from Jacobi's formula for the derivative of a determinant,

Fix an arbitrary $\tau>0$. Since any self-adjoint matrix is diagonalizable, we can choose an orthonormal basis $\{e_1,\cdots, e_{\dm-1}\}$ of $\gamma'(0)^\perp$ consisting of eigenvectors of $\mathcal{B}(\tau)$. Then consider the solution $\{\eta_1, \cdots, \eta_{\dm-1}\}$ of the Jacobi equation in $\gamma'(0)^\perp$, given by
\[
\eta_i(t)=\mathcal{A}(t)e_i, \qquad \text{ for all } 1\leq i\leq \dm-1. 
\]
It holds immediately that for all $t$,
\[
\langle \eta_i(t), \eta_j(t)\rangle=\langle \mathcal{A}(t)e_i,\mathcal{A}(t)e_j\rangle=\langle \mathcal{B}(t)e_i, e_j\rangle, \qquad 1 \leq i,j \leq \dm-1.
\]

By differentiating the above, we compute
\begin{align*}
\mathrm{tr} (\mathcal{B}'\mathcal{B}^{-1})(t)&=\sum_{i, j=1}^{\dm-1}\langle \mathcal{B}' (t)e_i, e_j\rangle \langle  e_j, \mathcal{B}^{-1}(t) e_i\rangle \\
&=\sum_{i, j=1}^{\dm-1}\left( \langle \eta_i'(t), \eta_j(t) \rangle+\langle \eta_i(t), \eta_j'(t) \rangle\right)\langle \mathcal{B}^{-1}(t)e_i, e_j\rangle.
\end{align*}
Using the assumption on the basis $\{e_1,\cdots, e_{\dm-1}\}$, $\langle \mathcal{B}(\tau)e_i, e_j\rangle= \lambda_i \delta_{ij}$ for some eigenvalues $\lambda_i$, $i = 1,\dots, \dm-1$. Also,  $\langle \mathcal{B}(\tau)^{-1} e_i, e_j\rangle\rangle=\lambda_i^{-1}\delta_{ij}$. Then, by evaluating at $t=\tau$, we find
\begin{align*}
\mathrm{tr}(\mathcal{B}'\mathcal{B}^{-1})(\tau)
&=\sum_{i, j=1}^{\dm-1}\left( \langle \eta_i'(\tau), \eta_j(\tau)\rangle+\langle \eta_i(\tau), \eta_j'(\tau)\rangle\right)\lambda_i^{-1} \delta_{ij}\\
&=2\sum_{i=1}^{\dm-1} \frac{\langle \eta_i'(\tau), \eta_i(\tau)\rangle}{\lambda_i}.
\end{align*}

Finally, we use
\[
\langle \eta_i(\tau),\eta_i(\tau)\rangle=\langle \mathcal{B}(\tau)e_i, e_i\rangle=\lambda_i, \qquad i =1, \dots, \dm-1,
\]
to conclude that
\[
 \frac{1}{2}\mathrm{tr}(\mathcal{B}'\mathcal{B}^{-1})(\tau)=\sum_{i=1}^{\dm-1} \frac{\langle \eta_i'(\tau), \eta_i(\tau)\rangle}{\langle \eta_i(\tau), \eta_i(\tau)\rangle}=\sum_{i=1}^{\dm-1} \frac{|\eta_i|'(\tau)}{|\eta_i(\tau)|}\geq (\dm-1)\frac{\psi_M'(\tau)}{\psi_M(\tau)},
\]
where for the last inequality we used \eqref{eqn:YpoY}.

Since $\tau$ was arbitrary, combining the above inequality with \eqref{eqn:detApodetA} we find
\begin{equation}
\label{eqn:l-ineq}
(\dm-1)\frac{\psi_M'}{\psi_M}\leq \frac{(\det\mathcal{A})'}{\det\mathcal{A}}.
\end{equation}
Inequality \eqref{eqn:l-ineq} can be integrated to get the desired conclusion. We explain this last step, together with the right inequality, in Part II.
\medskip


\noindent {\em Part II: Right inequality.} This part also follows closely the proof of the constant bounds case (see \cite[Theorem III.4.3]{Chavel2006}), but with some small extra technicalities. We present a lemma first.

\begin{lemma}
\label{lemma:Psi-ineq}
Let $\Psi$ satisfy the following ODE for some $\epsilon>0$:
\[
\displaystyle\Psi'(t)=(\dm-1)c(t)-\frac{\Psi^2(t)}{\dm-1},\qquad \textrm{ for all } t\geq\epsilon,\quad \Psi(\epsilon)>0,
\]
where $c(t)\geq0$ for all $t\geq \epsilon$. Then,
\[
\Psi(t)>0, \qquad \textrm{ for all } t>\epsilon.
\]
\end{lemma}
\begin{proof}
We have
\[
\Psi'(t)\geq -\frac{\Psi^2(t)}{\dm-1}, \qquad \text { for } t \geq \epsilon,
\]
which yields
\[
\frac{\d}{\d t}\left(\frac{1}{\Psi(t)}\right)\leq \frac{1}{\dm-1}.
\]

We can solve this differential inequality to get
\[
\frac{1}{\Psi(t)}-\frac{1}{\Psi(\epsilon)}\leq \frac{t - \epsilon}{\dm-1}, \qquad \text { for } t \geq \epsilon.
\]
Hence, we have
\[
\frac{1}{\Psi(t)}\leq \frac{1}{\Psi(\epsilon)}+\frac{t - \epsilon}{\dm-1}, \qquad \text { for all } t \geq \epsilon,
\]
which then implies the conclusion of the lemma.
\end{proof}

We define
\[
\Psi_m:=(\dm-1)\frac{\psi_m'}{\psi_m},\quad \Psi_M:=(\dm-1)\frac{\psi_M'}{\psi_M},\quad \Phi:=\frac{(\det\mathcal{A})'}{\det\mathcal{A}}.
\]
Using \eqref{eqn:psimM} we find that $\Psi_m$ and $\Psi_M$ satisfy the following equations:
\[
\Psi_m'=(\dm-1)c_m-\frac{\Psi_m^2}{\dm-1},\qquad \text{ and } \qquad \Psi_M'=(\dm-1)c_M-\frac{\Psi_M^2}{\dm-1}.
\]
Also, by the initial data in \eqref{eqn:psimM}, we have
\[
\Psi_m(t), \Psi_M(t)\sim \frac{\dm-1}{t}, \qquad \text { as } t\to 0^+.
\]

From Lemma \ref{lemma:Psi-ineq}, we know that both $\Psi_m$ and $\Psi_M$ are positive. From the left inequality (Part I), we have
\[
\Psi_M\leq \Phi,
\]
and this yields that $\Phi$ is also positive. 

By Jacobi's formula, we have
\[
\Phi=\mathrm{tr}\left( \mathcal{A}'\mathcal{A}^{-1} \right).
\]
We denote $\mathcal{U}:=\mathcal{A}'\mathcal{A}^{-1}$. From the argument in the proof of \cite[Theorem III.4.3]{Chavel2006}, it holds that
\[
(\mathrm{tr} \, \mathcal{U})'+\mathrm{tr}\, \mathcal{U}^2+\mathrm{tr}\, \mathcal{R}=0.
\]
Since $\mathcal{K}(x) \geq -c_m(\theta_x)$ and 
\[
\mathrm{tr}\, \mathcal{U}^2\geq\frac{(\mathrm{tr}\, \mathcal{U})^2}{\dm-1}, 
\]
we get
\[
\Phi'\leq (\dm-1)c_m-\frac{\Phi^2}{\dm-1}.
\]

We calculate now the time evolution of the functional $F:=1-\frac{\Phi}{\Psi_m}$:
\begin{align*}
-F'
&=\frac{\Phi'\Psi_m-\Phi\Psi_m'}{\Psi_m^2}\\
&\leq \frac{(\dm-1)c_m(\Psi_m-\Phi)-\frac{\Phi\Psi_m}{\dm-1}(\Phi-\Psi_m)}{\Psi_m^2}\\
&= \frac{(\dm-1)c_m}{\Psi_m}\left(1-\frac{\Phi}{\Psi_m}\right)+\frac{\Phi}{\dm-1}\left(1-\frac{\Phi}{\Psi_m}\right)\\
&=\left(
\frac{(\dm-1)c_m}{\Psi_m}+\frac{\Phi}{\dm-1}
\right)F.
\end{align*}
This yields
\begin{equation}
\label{eqn:F-ineq}
F'\geq -\left(
\frac{(\dm-1)c_m}{\Psi_m}+\frac{\Phi}{\dm-1}
\right)F.
\end{equation}

We also find
\[
\lim_{t\to0+}\frac{\Phi(t)}{\Psi_m(t)}=\lim_{t\to0+}\frac{(\det\mathcal{A})'(t)\psi_m(t)}{(\dm-1)\det\mathcal{A}(t) \psi_m'(t)}=\lim_{t\to0+}\left(\frac{t(\det\mathcal{A})'(t)}{(\dm-1)\det\mathcal{A}(t)}\right)\left(
\frac{\psi_m(t)}{\psi_m'(t)t}
\right)=1.
\]
Hence, we have $\lim_{t\to0+}F(t)=0$, and we can define $F(0)=0$. Let assume that for some $T>0$, $F(T)<0$. Then, we can define $0\leq T_0<T$ which satisfies
\[
T_0=\sup\{t:F(t)=0\}.
\]
Then,
\[
F(t)\leq 0, \qquad \quad\forall~T_0\leq t\leq T,
\]
which by \eqref{eqn:F-ineq}, it yields
\[
F'\geq0, \qquad\forall~T_0\leq t\leq T.
\]
Finally, we get
\[
F(T)-F(T_0)=\int_{T_0}^TF'(t)\d t\geq0.
\]
However, we assumed $F(T_0)=0$ and $F(T)<0$, so we have a contradiction. We conclude that $F(t)\geq0$ for all $t\geq0$, or equivalently,
\[
\Phi(t)\leq \Psi_m(t), \qquad\forall~t>0.
\]
Using the expressions of $\Phi$ and $\Psi$, we then have
\begin{equation}
\label{eqn:r-ineq}
\frac{(\det\mathcal{A})'}{\det\mathcal{A}}\leq (\dm-1)\frac{\psi_m'}{\psi_m}.
\end{equation}

Now, by \eqref{eqn:l-ineq} and \eqref{eqn:r-ineq}, we have both inequalities:
\[
(\dm-1)\frac{\psi_M'}{\psi_M}\leq \frac{(\det \mathcal{A})'}{\det\mathcal{A}}\leq (\dm-1)\frac{\psi_m'}{\psi_m}.
\]
We integrate this double inequality from $\epsilon>0$ to $\theta$ to get
\begin{equation}
\label{eqn:double-ineq}
\psi_M(\theta)^{\dm-1} \left(\frac{\det\mathcal{A}(\epsilon)}{\psi_M(\epsilon)^{\dm-1}}\right)\leq \det\mathcal{A}(\theta)\leq \psi_m(\theta)^{\dm-1} \left(\frac{\det\mathcal{A}(\epsilon)}{\psi_m(\epsilon)^{\dm-1}}\right).
\end{equation}

Since $\mathcal{A}'(0)=I$ and $\mathcal{A}(0)=0$, we have
\[
\lim_{\epsilon\to0+}\frac{\det\mathcal{A}(\epsilon)}{\epsilon^{\dm-1}}=1.
\]
Together with the initial conditions for $\psi_m$ and $\psi_M$ (see \eqref{eqn:psimM}), this further yields
\[
\lim_{\epsilon\to0+}\frac{\det\mathcal{A}(\epsilon)}{\psi_m(\epsilon)^{\dm-1}}=\lim_{\epsilon\to0+}\frac{\det\mathcal{A}(\epsilon)}{\psi_M(\epsilon)^{\dm-1}}=1.
\]
Finally, we let $\epsilon\to0^+$ in \eqref{eqn:double-ineq} to get
\[
\psi_M(\theta)^{\dm-1} \leq \det\mathcal{A}(\theta)\leq \psi_m(\theta)^{\dm-1}.
\]
This concludes the proof of Theorem \ref{lemma:Chavel-thms-gen}.


\section{Proof of Lemma \ref{estpsi}}
\label{appendix:estpsi}
In this appendix, we prove each part of Lemma \ref{estpsi}.
\smallskip

\underline{\em Proof of part I.} In this part, we only assume that $c$ is positive, continuous, and non-decreasing, and show \eqref{estpsires-relaxed}.

Denote
\begin{equation}
\label{eqn:G}
G(\theta)=\frac{\psi(\theta)}{\psi'(\theta)},\qquad \text{ for } \theta \geq 0.
\end{equation}
Then, by direct computations, 
\begin{align}
G'(\theta)&=\frac{\psi'(\theta)\psi'(\theta)-\psi(\theta)\psi''(\theta)}{\psi'(\theta)^2} \nonumber \\[2pt]
&=1-c(\theta) G(\theta)^2.
\label{eqn:Gp}
\end{align}
Note that $G(0)=\frac{\psi(0)}{\psi'(0)}=0$ and $G'(0)=1-c(0)G(0)^2=1$. 

We will show that 
\begin{equation*}
\liminf_{\theta\to \infty}\sqrt{c(\theta)}G(\theta)\geq1.
\end{equation*}
Fix $\bar{\theta}>0$. Since $c$ is non-decreasing, it holds that
\[
G'(\theta)\geq 1-c(\bar{\theta})G(\theta)^2, \qquad\forall0\leq \theta\leq\bar{\theta}.
\]

We consider the following two cases.

\noindent (Case 1) If there exists $s\in [0, \bar{\theta}]$ such that $G(s)=\frac{1}{\sqrt{c(\bar{\theta})}}$, then we have
\[
G(\theta)\geq \frac{1}{\sqrt{c (\bar{\theta})}}, \qquad\forall s \leq \theta\leq \bar{\theta},
\]
since $G'(\theta)\geq 0$ whenever $G(\theta)=\frac{1}{\sqrt{c(\bar{\theta})}}$. In particular,
\[
G(\bar{\theta})\geq \frac{1}{\sqrt{c(\bar{\theta})}}.
\]

\noindent (Case 2) If $G(\theta)<\frac{1}{\sqrt{c(\bar{\theta})}}$ for all $0\leq\theta\leq \bar{\theta}$, then $G'(\theta)>0$ for all $0\leq \theta\leq \bar{\theta}$. Furthermore,
\[
\frac{G'(\theta)}{1-c(\bar{\theta})G(\theta)^2}\geq1, \qquad\forall0\leq \theta \leq\bar{\theta},
\]
and hence,
\[
\frac{\d}{\d\theta}\mathrm{atanh} \Bigl(\sqrt{c(\bar{\theta})}G(\theta)\Bigr)\geq \sqrt{c(\bar{\theta})},\qquad\forall 0\leq \theta\leq \bar{\theta}.
\]
By integrating from $0$ to $\bar{\theta}$ and using $G(0)=0$ we then get
\[
G(\bar{\theta})\geq  \frac{\tanh \left(\sqrt{c(\bar{\theta})}\bar{\theta} \right)}{\sqrt{c(\bar{\theta})}}.
\]

From Case 1 and Case 2, we conclude that
\[
G(\bar{\theta}) \geq \min\left(\frac{1}{\sqrt{c(\bar{\theta})}}, \frac{\tanh \left( \sqrt{c(\bar{\theta})}\bar{\theta} \right)}{\sqrt{c(\bar{\theta})}}\right).
\]
Then, given that $ \min\left(1, \tanh \left(\sqrt{c(\bar{\theta})}\bar{\theta} \right)\right) = \tanh \left(\sqrt{c(\bar{\theta})}\bar{\theta} \right)$, we find
\[
\sqrt{c(\bar{\theta})}G(\bar{\theta})\geq \tanh \left(\sqrt{c(\bar{\theta})}\bar{\theta} \right).
\]

We multiply $\sqrt{c(\bar{\theta})}$ to the above inequality to get
\begin{equation*}
\sqrt{c(\bar{\theta})}G(\bar{\theta})\geq \min\left(1, \tanh \left(\sqrt{c(\bar{\theta})}\bar{\theta} \right)\right)=\tanh \left (\sqrt{c(\bar{\theta})}\bar{\theta} \right).
\end{equation*}
From here, since the choice of $\bar{\theta}$ was arbitrary, we get
\[
\liminf_{\theta\to\infty}\sqrt{c(\theta)}G(\theta)\geq \lim_{\theta\to\infty}\tanh(\sqrt{c(\theta)}\theta)=1.
\]
Finally, use the above and the expression \eqref{eqn:G} for $G$ to infer that for any $\epsilon>0$, there exists $\theta_0>0$ such that
\[
\frac{1}{1+\epsilon}\leq \frac{\sqrt{c(\theta)}\psi(\theta)}{\psi'(\theta)},\qquad \forall\theta\geq\theta_0.
\]
By integration, this implies \eqref{estpsires-relaxed} (also note the properties of $\psi$ from Remark \ref{rmk:psimM}).\\
\medskip

\underline{ \em Proof of part II.}
In this part we further assume that $c$ satisfies \eqref{eqn:c32}, and show \eqref{estpsires}.  We will use the function $G(\cdot)$ defined in \eqref{eqn:G}, and show
\begin{equation}
\label{eqn:claim2}
\limsup_{\theta\to\infty}\sqrt{c(\theta)}G(\theta)\leq 1.
\end{equation}

We consider two cases again.\\

\noindent (Case 1) There exists no $\theta\geq0$ such that $G'(\theta)=0$. In this case, given the expression of $G'$ in \eqref{eqn:Gp}, there is no $\theta\geq0$ such that
\[
G(\theta)=\frac{1}{\sqrt{c(\theta)}}.
\]
Since $G(0)=0$, we infer that $G(\theta)<\frac{1}{\sqrt{c(\theta)}}$ for all $\theta\geq0$. Indeed, by assuming the contrary, one can apply the intermediate value theorem to $G(\theta)-\frac{1}{\sqrt{c(\theta})}$ to get a contradiction. Hence,
\[
\sqrt{c(\theta)}G(\theta)<1,\qquad\forall \theta\geq0,
\] 
which yields \eqref{eqn:claim2}.

\noindent (Case 2) Now assume that there exists $\theta\geq0$ such that $G'(\theta)=0$. We define
\[
\mathcal{I}=\{x\geq0: G(x)=1/\sqrt{c(x)}\},
\]
or equivalently (see \eqref{eqn:Gp}), $\mathcal{I}$ consists in the set of zeros of $G'$. Since $G(\theta)-\frac{1}{\sqrt{c(\theta)}}$ is a continuous function, $\mathcal{I}$ is a closed set. So, if we define 
\[
\theta_1:=\inf\mathcal{I},
\]
then $\theta_1\in \mathcal{I}$. We will show that $\theta_1$ is an isolated point. By simple manipulations, we compute
\begin{align*}
\overline{D} \biggl( \frac{1}{\sqrt{c(\theta)}} \biggr) &=\limsup_{h\to0}\frac{\frac{1}{\sqrt{c(\theta+h)}}-\frac{1}{\sqrt{c(\theta)}}}{h} \\
&=\limsup_{h\to0}\left(-\frac{c(\theta+h)-c(\theta)}{h}\cdot  \frac{1}{\sqrt{c(\theta)}\sqrt{c(\theta+h)}(\sqrt{c(\theta)}+\sqrt{c(\theta+h)})}\right)\\
&=-\liminf_{h\to0}\left(\frac{c(\theta+h)-c(\theta)}{h}\right)\frac{1}{2c(\theta)\sqrt{c(\theta)}}\\
&=-\frac{1}{2 c(\theta)^{3/2}}\underline{D}c(\theta).
\end{align*}

Since $\underline{D}c(\theta)>0$ for all $\theta \geq 0$, we infer that
\[
\left( \overline{D} \frac{1}{\sqrt{c}} \right) (\theta)<0,\qquad \text{ for all } \theta \geq 0.
\]
Let $m_1:=\left(\overline{D}\frac{1}{\sqrt{c}}\right)(\theta_1)<0$. Then, there exists $\epsilon_1>0$ such that 
\[
\frac{1}{\sqrt{c(\theta)}}< \frac{1}{\sqrt{c(\theta_1)}}+\frac{m_1}{2}(\theta-\theta_1), \qquad \forall \theta\in (\theta_1, \theta_1+\epsilon_1).
\]
Since $G'(\theta_1)=0$, we also know that there exists $\epsilon_2>0$ such that
\[
G(\theta)> G(\theta_1)+\frac{m_1}{2}(\theta-\theta_1), \qquad \forall \theta\in (\theta_1, \theta_1+\epsilon_2).
\] 
Hence, if we set $\epsilon=\min(\epsilon_1, \epsilon_2)$, and combine the two inequalities above, we get
\begin{equation}
\label{eqn:G-lb1}
G(\theta) >\frac{1}{\sqrt{c(\theta)}}, \qquad \forall \theta\in(\theta_1, \theta_1+\epsilon),
\end{equation}
where we also used that $G(\theta_1)=\frac{1}{\sqrt{c(\theta_1)}}$. We conclude that $\mathcal{I} \cap (\theta_1, \theta_1+\epsilon) = \emptyset$. By combining this with $\theta_1=\inf\mathcal{I}$ we get
\[
\mathcal{I}\backslash\{\theta_1\}\cap(\theta_1-\epsilon, \theta_1+\epsilon)=\emptyset,
\] 
which implies that $\theta_1$ is an isolated point of $\mathcal{I}$. 

Next, we want to show that $\theta_1$ is the only element of $\mathcal{I}$. If $\mathcal{I}$ contains an element other than $\theta_1$, then $\mathcal{I}\backslash \{\theta_1\}$ is a non-empty closed set. Define
\[
\theta_2:=\inf(\mathcal{I}\backslash\{\theta_1\}),
\]
which satisfies $\theta_2\in\mathcal{I}$ and $\theta_1<\theta_2$. From \eqref{eqn:G-lb1}, we infer that
\begin{equation}
\label{eqn:G-lb2}
G(\theta)>\frac{1}{\sqrt{c(\theta)}}, \qquad \forall \theta\in(\theta_1, \theta_2). 
\end{equation}

Define $m_2:=\left(\overline{D}\frac{1}{\sqrt{c}}\right)(\theta_2)<0$. Then there exists $0<\epsilon_3<\theta_2-\theta_1$ such that
\[
\frac{1}{\sqrt{c(\theta)}}>\frac{1}{\sqrt{c(\theta_2)}}+\frac{m_2}{2}(\theta-\theta_2), \qquad \forall \theta\in (\theta_2-\epsilon_3, \theta_2).
\]
Since $G'(\theta_2)=0$, there exists $0<\epsilon_4<\theta_2-\theta_1$ such that
\[
G(\theta)<G(\theta_2)+\frac{m_2}{2}(\theta-\theta_2), \qquad \forall \theta\in (\theta_2-\epsilon_4, \theta_2).
\]
Now define $\epsilon'=\min(\epsilon_3, \epsilon_4)$ and combine the two inequalities above to get
\[
G(\theta)< \frac{1}{\sqrt{c(\theta)}}, \qquad \forall \theta\in (\theta_2-\epsilon', \theta_2). 
\]
where we also used $G(\theta_2)=\frac{1}{\sqrt{c(\theta_2)}}$. However, this contradicts \eqref{eqn:G-lb2}. We conclude that $\mathcal{I}\backslash \{\theta_1\}$ is an empty set and $\theta_1$ is the only element of $\mathcal{I}$. Also,
\begin{align}
\label{estF}
G(\theta)>\frac{1}{\sqrt{c(\theta)}},\qquad \forall \theta>\theta_1.
\end{align}

Fix $0<\delta<2$ arbitrary. By the assumption \eqref{eqn:c32} on $c(\theta)$, there exists $\tilde{\theta}_1>0$ such that
\begin{equation}
\label{eqn:uD-c32}
\frac{\overline{D}c(\theta)}{c(\theta)^{3/2}}<\delta, \qquad \forall \theta>\tilde{\theta}_1.
\end{equation}
Now, set $\tilde{\theta}=\max(\theta_1, \tilde{\theta}_1)$. From \eqref{eqn:Gp} and \eqref{estF}, we infer
\begin{equation}
\label{eqn:Gp-ub1}
G'(\theta) \leq 1-\sqrt{c(\theta)}G(\theta),\qquad\forall \theta> \tilde{\theta}.
\end{equation}
Also define
\[
H(\theta)=\exp\left(\int_0^\theta \sqrt{c(s)}\d s\right).
\]
By the subadditivity of limit superior, \eqref{eqn:Gp-ub1} and $H'(\theta) = \sqrt{c(\theta)} H(\theta)$, we get 
\begin{align*}
& \overline{D} \left(H(\theta)(\sqrt{c(\theta)}G(\theta)-1) \right)\leq \overline{D}H(\theta) \, \left(\sqrt{c(\theta)}G(\theta)-1 \right)+H(\theta)\frac{\overline{D}c(\theta)}{2\sqrt{c(\theta)}}G(\theta)
+H(\theta)\sqrt{c(\theta)} \,\overline{D}G(\theta) \\[2pt]
&\qquad \leq \sqrt{c(\theta)}H(\theta)(\sqrt{c(\theta)}G(\theta)-1)+H(\theta)\frac{\overline{D}c(\theta)}{2\sqrt{c(\theta)}} G(\theta)+H(\theta)\sqrt{c(\theta)}(1-\sqrt{c(\theta)}G(\theta))\\
& \qquad =H(\theta)G(\theta)\frac{\overline{D}c(\theta)}{2\sqrt{c(\theta)}},
\end{align*}
for $\theta >  \tilde{\theta}$. Then, also using \eqref{eqn:uD-c32}, we derive
\begin{align*}
\overline{D} \left(H(\theta)(\sqrt{c(\theta)}G(\theta)-1) \right) &=H(\theta)G(\theta)c(\theta)\frac{\overline{D}c(\theta)}{2c(\theta)^{3/2}}\\
&\leq \frac{\delta}{2}H(\theta)G(\theta)c(\theta)\\
&= \frac{\delta}{2}H(\theta) \left(\sqrt{c(\theta)} G(\theta)-1 \right)\sqrt{c(\theta)}+\frac{\delta}{2}H(\theta)\sqrt{c(\theta)},
\end{align*}
for $\theta> \tilde{\theta}$.

Multiply the above by $H(\theta)^{-\frac{\delta}{2}}$, and use again $H'(\theta) = \sqrt{c(\theta)} H(\theta)$ to get
\begin{align*}
\overline{D}\left(H(\theta)^{1-\delta/2} \left(\sqrt{c(\theta)}G(\theta)-1 \right)\right) & \leq \frac{\delta}{2}H^{1-\delta/2}(\theta)\sqrt{c(\theta)} \\
&=\frac{\delta}{2-\delta}\left(H^{1-\delta/2}(\theta)\right)',
\end{align*}
for $\theta>\tilde{\theta}$. This implies 
\[
H(\theta)^{1-\delta/2} \left(\sqrt{c(\theta)} G(\theta)-1 \right)-H(\tilde{\theta})^{1-\delta/2} \left(\sqrt{c(\tilde{\theta})} G(\tilde{\theta})-1 \right)\leq \frac{\delta}{2-\delta}\left(H(\theta)^{1-\delta/2}-H(\tilde{\theta})^{1-\delta/2}\right),
\]
for $\theta>\tilde{\theta}$.

Finally, we get
\[
\sqrt{c(\theta)} G(\theta) -1\leq \frac{H(\tilde{\theta})^{1-\delta/2}\left(\sqrt{c(\tilde{\theta})} G(\tilde{\theta})-1\right)}{H(\theta)^{1-\delta/2}}+\frac{\delta}{2-\delta}\left(1-\frac{H(\tilde{\theta})^{1-\delta/2}}{H(\theta)^{1-\delta/2}}\right),
\]
for $\theta> \tilde{\theta}$, and since
\[
\lim_{\theta\to\infty}H(\theta)=\infty,
\]
we further find
\[
\limsup_{\theta\to\infty}\sqrt{c(\theta)} G(\theta) \leq 1+\frac{\delta}{2-\delta}.
\]
Given that $\delta$ is arbitrary, we can now conclude \eqref{eqn:claim2}. From \eqref{eqn:claim2} and the expression \eqref{eqn:G} for $G$, for any $0<\epsilon<1$ there exists $\theta_0>0$ such that
\[
\frac{\sqrt{c(\theta)}\psi(\theta)}{\psi'(\theta)}\leq \frac{1}{1-\epsilon}, \qquad \forall \theta\geq\theta_0.
\]
By integration, this implies \eqref{estpsires}. This concludes the proof of Lemma \ref{estpsi}.


\bibliographystyle{abbrv}
\def\url#1{}
\bibliography{lit.bib}

\end{document}